\documentclass{article}

\usepackage{amsmath,amsthm,amsopn,amsfonts,graphicx,amssymb,dsfont,color,hyperref}
\usepackage{gensymb} 
\usepackage{textcomp}
\usepackage{mathrsfs}
\usepackage{color}
\usepackage{geometry}
\usepackage{subfigure}

\newtheorem{theorem}{Theorem}[section]
\newtheorem{prop}[theorem]{Proposition}
\newtheorem{defi}[theorem]{Definition}
\newtheorem{lem}[theorem]{Lemma}
\newtheorem{rem}[theorem]{Remark}

\newcommand{\baco}{\left\{ \begin{array}}
\newcommand{\eaco}{\end{array} \right.}

 \newcommand{\ot}{\overline \theta}
\newcommand{\oc}{\overline c}
\newcommand{\ol}{\overline \ell}

\newcommand{\ds}{\displaystyle}
\def\lp {\left( }
\def\rp {\right) }
\def\p{\partial}

\def\R{\mathbb R}
\def\N{\mathbb N}

\def\W{\mathcal W}
\def\opt{\mathcal{O}}

\def\ep{\varepsilon}

\def\j{\mathbf j}
\def\u{\mathbf{u}}

\def\x{\mathbf{x}}
\def\y{\mathbf{y}}
\def\z{\mathbf{z}}
\def\l{\mathbf{l}}
\def\0{\mathbf 0}

\def\tg{\tilde{\gamma}}

\def\xb{\overline{\x}}
\def\yb{\overline{\y}}

\def\k{{\mathbf k}}
\def\w{{\mathbf w}}
\def\1{{\mathbf 1}}

\definecolor{aquamarine}{rgb}{0.13, 0.68, 0.8}

\definecolor{colflo}{rgb}{0.7, 0.5, 1}

\def\Bk{\color{black}}

\title{\bf A host-pathogen coevolution model. Part I:  Run straight for your life 
}

\begin{document}
\maketitle

\begin{center}
{\large \bf Matthieu Alfaro\footnote{Univ. Rouen Normandie, CNRS, LMRS UMR 6085, F-76000 Rouen, France.}, Florian Lavigne\footnote{Univ. Rouen Normandie, CNRS, LMRS UMR 6085, F-76000 Rouen, France.}, and Lionel Roques\footnote{INRAE, BioSP, 84914, Avignon, France.}}
\end{center}

\vspace{15pt}


\vspace{10pt}

\begin{abstract}
In this study, we propose a novel model describing the coevolution between hosts and pathogens, based on a non-local partial differential equation formalism for populations structured by phenotypic traits. Our objective with this model is to illustrate scenarios corresponding to the evolutionary concept of ``Chase Red Queen scenario", characterized by perpetual evolutionary chases between hosts and pathogens. 
First, numerical simulations show the emergence of such scenarios, depicting the escape of the host (in phenotypic space) pursued by the pathogen. We observe two types of behaviors, depending on the assumption about the presence of a phenotypic optimum for the host: either the formation of traveling pulses moving along a straight line with constant speed and constant profiles, or stable phenotypic distributions that periodically rotate along a circle in the phenotypic space. Through rigorous perturbation techniques and careful application of the implicit function theorem in rather intricate function spaces, we demonstrate the existence of the first type of behavior, namely traveling pulses moving with constant speed along a straight line. Just as the Lotka-Volterra models have revealed periodic dynamics without the need for environmental forcing, our work shows that, from the pathogen's point of view, various trajectories of mobile optima can emerge from coevolution with a host species.\\


\end{abstract}

\vspace{3pt}

\section{Introduction}\label{s:intro}

In Lewis Carroll's novel \textit{Through the Looking-Glass} \cite{Car-1872}, the Red Queen says the famous sentence: ``It takes all the running you can do, to keep in the same place". This metaphor was later adopted by Van Valen, see \cite{Van73}, to formulate his evolutionary theory (extensively reviewed in \cite{BroChap14}):
\begin{quote}
    ``The Red Queen does not need changes in the physical environment, although she can accommodate them. Biotic forces provide the basis for a self-driving... perpetual motion of the effective environment and so of the evolution of the species affected by it."
\end{quote}
In the framework of host-pathogen interactions, this theory, known as the ``Red Queen Hypothesis", posits that  pathogens apply evolutionary pressure on hosts to develop resistance, while simultaneously evolving to sustain their infectivity   \cite{Sir-et-al-15}. These interactions can manifest in three distinct scenarios, summarized in \cite{BroChap14}:
\begin{itemize}
    \item Fluctuating Red Queen: this scenario describes how allele frequencies within a population oscillate over time. Predators, parasites, or competitors target the most common genotypes, providing an opportunity for rarer genotypes to flourish. This cyclical pattern ensures genetic diversity is maintained as environmental conditions and selective pressures change \cite{Bell19}. 
    \item Escalatory Red Queen: here, species are engaged in an evolutionary arms race, constantly adapting to outdo each other. Each new adaptation by one species prompts a counter-adaptation by its competitors or predators, leading to a continuous cycle of escalation. This process drives significant evolutionary changes as species strive to surpass one another's adaptations \cite{DawKreb79,NuiRid07,Sas00}.
    \item Chase Red Queen:  this scenario focuses on the evolutionary chase between  pathogens and hosts. Hosts evolve to become less exploitable, while pathogens evolve to counter these adaptations, aiming to reduce the phenotypic gap. This results in a perpetual cycle of adaptation and counter-adaptation, with neither side gaining a lasting upper hand \cite{Gav97,Kopp06}.
\end{itemize}

Several modeling approaches have been proposed to address the coevolutionary dynamics in host-pathogen interactions. These range from simple genetic models focusing on one or two loci to sophisticated simulations incorporating population dynamics, quantitative traits, and complex genetic structures \cite{BucAsh22}. Alongside this, numerous single-species PDE models of populations structured by traits have recently emerged  \cite{AlfCar17,ChaFer06,HamLav20,HamLavRoq20,MirGan20}. These models track the dynamics of the distribution of phenotypic traits described by a vector $\mathbf{x} \in \mathbb{R}^n$ over time, influenced by mutations and selection. Most of these models assume a unique phenotypic optimum, with fitness decreasing as one deviates from this optimum, following the paradigm of Fisher's Geometrical Model. Following this paradigm, a first objective of this work is to develop and analyze a host-pathogen coevolution PDE model for asexual populations structured by traits, aiming to encapsulate the principal aspects of their interactions. We particularly focus on the emergence of the Chase Red Queen scenario with this formalism. Another goal is to justify the existence of a shifting phenotypic optimum over time from the pathogen's perspective. Indeed, several studies have recently focused on analyzing single-species PDE models with a shifting phenotypic optimum that either moves at a constant speed \cite{AlfBer17,CalHen22}, fluctuates periodically \cite{CarNad20,FigMir19,LorChi15}, or exhibits general dynamics \cite{Lav23,RoqPat20}. In all these works, the trajectory of the phenotypic optimum was a given. Our goal is to demonstrate that, just as the Lotka-Volterra models have revealed periodic dynamics without the need for environmental forcing, these various trajectories of mobile optima can emerge from coevolution with a host species.

\section{A new host-pathogen coevolution model}\label{s:model}

\paragraph{Demographic model.} 
We model the dynamics of the host population size using a logistic growth term, complemented by a Lotka-Volterra-like term to represent the impact of the pathogen on the host compartment, 
\[
\frac{dH}{dt} = r_H H - \gamma_H H^2 - \rho HP, \quad t > 0,
\]
where $H(t)$ and $P(t)$ respectively denote the host and pathogen population sizes at time $t$. Here, $r_H$ represents the host's intrinsic growth rate, $\gamma_H > 0$ is a constant measuring the effect of intraspecific competition, and $\rho > 0$ quantifies the pathogen's impact on the host growth rate. On the other hand, the pathogen population grows at a rate $r_P > 0$ but is constrained by the number of hosts, implying that the carrying capacity for $P$ is determined by $\frac{r_P H(t)}{\gamma_P}$, as follows:
\[
\frac{dP}{dt} = r_P P - \gamma_P \frac{P^2}{H}, \quad t > 0.
\]
For any initial data $H(0)>0$, $P(0)>0$, the long time behavior is coexsitence of both populations:
\[
    \lim_{t \to +\infty} P(t) = \frac{r_H r_P}{\gamma_H \gamma_P + \rho r_P} \quad \text{and} \quad \lim_{t \to +\infty} H(t) = \frac{r_H \gamma_P}{\gamma_H \gamma_P + \rho r_P}.
    \]

\paragraph{Phenotypically structured population model.} We now extend the model to include that both the host and pathogen populations are structured by their respective phenotypes, denoted as $\mathbf{x} = (x_1, \dots, x_m) \in \mathbb{R}^m$ for the host and $\mathbf{y} = (y_1, \dots, y_n) \in \mathbb{R}^n$ for the pathogen. Here, each coordinate $x_i$ for the host (and correspondingly, $y_j$ for the pathogen) represents a distinct biological trait. For simplicity, we posit that the number of biological traits of interest is the same for both the host and the pathogen,  that is
\[
n = m.
\]

Let $ h(t,\mathbf x)$ be the  density  of hosts with phenotype $\mathbf x\in \R^n$  and $ p(t, \mathbf y)$  the density  of pathogens with phenotype $\mathbf y \in \R^n$ at time $t\ge 0$. The coevolutionary dynamics of the host and pathogen populations are described by the system
\begin{equation}  \label{sys:initial}
    \left\{\begin{array}{ll}
        \ds  \p_t h=\mu_H^2 \Delta_\x h +r_H[h,\x] h-\gamma_H H h - \rho[p,\x]  h P, & \quad t>0,\,  \mathbf x \in \R^n, \vspace{7pt}\\
         \ds \p_t p = \mu_P^2 \Delta_\y p + r_P[h,\y]p - \gamma_P \, \frac{P}{H}\, p, & \quad t>0,\,  \mathbf y\in \R^n,
    \end{array}\right.
\end{equation}
where  $H=H(t)$ and $P=P(t)$ are the total populations of hosts and pathogens, that is
\[
 H(t):= \int_{\R^n} h(t,\mathbf x) \, \mathrm d\mathbf x\qquad  \hbox{ and }\qquad  P(t):= \int_{\R^n } p(t,\mathbf y)\, \mathrm d\mathbf y.
\]
The Laplacian terms in \eqref{sys:initial} model the effects of mutations, that act on $\mathbf x$ for the hosts and on $\mathbf y$ for the pathogens. The coefficients $\mu_H^2$ and $\mu_P^2$ measure the intensity of the mutations (i.e., their strength and rate, under a weak selection strong mutation assumption, see the Appendix in \cite{HamLav20}).

By formally integrating \eqref{sys:initial} over $\mathbb{R}^n$, we obtain a system that is closely related to our initial ODE system, namely
\begin{equation}  \label{sys:ODE2}
\left\{\begin{array}{l}
        \ds \frac{dH}{dt} = \overline{r}_H(t) H - \gamma_H H^2 - \overline{\rho}(t) HP,  \vspace{7pt}\\
         \ds \frac{dP}{dt} = \overline{r}_P(t) P - \gamma_P \frac{P^2}{H}, 
    \end{array}\right.
\end{equation}
where $\overline{r}_H(t)$ represents the mean fitness of the host population, $\overline{r}_P(t)$ denotes the mean fitness of the pathogen population, and $\overline{\rho}(t)$ is the mean value of $\rho[p,\x]$ within the host population.

\paragraph{Fitness functions.} To define the fitness functions $r_H$ and $r_P$, we draw inspiration from Fisher's Geometric Model. More precisely, each fitness function has a unique maximum and decreases quadratically from it. 

The function $r_H[h,\mathbf{x}]$ describes the fitness of the hosts with phenotype $\mathbf{x}$ in the absence of pathogens:
\begin{equation}
    r_H[h,\x]=R_H-\alpha_H^2 \|\x\|^2-\beta^2 \|\x-\xb(t)\|^2, 
\label{def:rH}
\end{equation}
with $R_H>0$ the fitness of the optimal phenotype. When $\alpha_H > 0$, the term $\alpha_H \|\mathbf{x}\|^2$ indicates that, all else being equal, phenotypes $\mathbf{x}$ close to $\mathbf{0}$ tend to have higher fitnesses. As for the less usual term involving $\beta$ and the mean host phenotype
\begin{equation}
    \xb(t) := \frac{1}{H(t)} \int_{\R^n} \x\, h(t,\x)\, \mathrm d\x,
\label{def:xb}
\end{equation}
it can be considered as a ``concerted evolution'' term, modeling selection around the mean phenotype, thus preventing excessive variance in the distribution of $h(t,\mathbf{x})$ over time. We will see later that, for the model \eqref{sys:initial} presented here, this term plays a crucial role in achieving trajectories that describe a Chase Red Queen scenario.

For simplicity, we assume that the pathogen phenotype optimum only depends on the host phenotype distribution through the mean host phenotype $\xb(t)$: 
\begin{equation}
    r_P[h,\y]=R_P-\alpha_P^2 \|\y-\opt(\xb(t))\|^2,
\label{def:rP}
\end{equation}
where $R_P>0$ represents the fitness of the optimal phenotype, $\alpha_P >0$ is the selection pressure on pathogens, and $\opt:\R^n\to \R^n$ is a function to be specified, describing how the pathogen phenotype optimum is influenced by the mean host phenotype.

The second equation in \eqref{sys:initial} is thus recast
\begin{equation}
\p_tp=\mu_P^2\Delta_\y p+\left(R_P-\alpha_P^2\|\y-\opt(\xb(t))\|^2 \right)p-\, \gamma_P \, \frac{P}{H}\, p,\quad t>0,\, \y\in \R^n.
    \label{eq:opt_mob}
\end{equation}
If $\opt(\xb(t))$ were independent of hosts and defined as a given function $\opt (t)$, \eqref{eq:opt_mob} could be considered as a moving optimum problem. As mentioned above, this type of problem has been studied for several types of environmental fluctuations. For instance, the constant speed assumption in \cite{AlfBer17,CalHen22}, corresponds to  $\opt(t)=\opt_0+ct\u$ for some speed $c$ and unit vector $\u\in \R^n$. In previous studies, periodicity was considered along a line: $\opt(t)= f(t)\u$ for some scalar periodic function $f$, see  \cite{CarNad20,FigMir19,LorChi15,RoqPat20}. We will explore here how this type of mobile optimum, or generalizations where the optimum does not necessarily remain on a line, can emerge from interactions with the host.

\paragraph{Pathogen's impact.}  The function $\rho[p,\x]$  measuring the strength of the impact of the pathogen on the host in the first equation of \eqref{sys:initial} has to remain positive so we use a Gaussian profile, namely
\begin{equation}
    \rho[p,\x]=\rho_{\max} \, e^{-\, \theta \|\x-\W(\yb(t))\|^2},
\label{def:rho}
\end{equation}
where $\rho_{\max}>0$ is the maximal impact, $\theta>0$ and $\W:\R^n\to \R^n$ a function describing the host ``worst" phenotype, which is the most sensitive to the pathogen. We assume that it only depends on the mean pathogen phenotype $\yb(t)$:
\begin{equation}
    \yb(t):=\frac{1}{P(t)} \int_{\R^n}\y \ p(t,\y)\ \mathrm d \y.
    \label{def:yb}
\end{equation}

\paragraph{The full model.}  
For biological realism, host and pathogen phenotypes should not be comparable except through complex functions $\opt$ and $\W$. Nevertheless, for simplicity and to allow for a rigorous mathematical study, we assume here that the pathogen's phenotypic optimum is a translation of the host's mean phenotype, while the worst phenotype for the host corresponds to the mean phenotype of the pathogen.
Thus, from now on, for $\ell\geq  0$ and a given unit vector $\u$, we choose
$$
\opt(\x)=\x-\ell\u, \quad \W(\y)=\y,
$$
and we focus on
\begin{equation}  \label{sys:full}
    \left\{\begin{array}{ll}
        \ds \p_t h=\mu_H^2 \Delta_\x h +\left(R_H-\gamma_H H -\alpha_H^2 \|\x\|^2-\beta^2\|\x-\xb(t)\|^2- P\rho_{\max} \, e^{-\, \theta \|\x-\yb(t)\|^2} \right)  h , \vspace{7pt} \\
        \ds \p_tp = \mu_P^2 \Delta_\y p + \left(R_P- \gamma_P \, \frac{P}{H}-\alpha_P^2 \|\y+\ell\u -\xb(t)\|^2\right)p.
    \end{array}\right.
\end{equation}
In the present work, we focus on solutions corresponding to a demographic equilibrium,  that is $H(t)=cste=H$ and $P(t)=cste=P$.

\paragraph{Organization of the paper.} System \eqref{sys:full} is the starting point of the present work. We aim at performing a rigorous analysis of  such nonlocal PDE systems, thus shedding light on the aforementioned biological scenarii.

We start with some numerical explorations of \eqref{sys:full} in Section \ref{s:num}, revealing different outcomes depending on parameters $\alpha_H$ and $\beta$. In Section \ref{s:static}, we consider the case $\beta=0$ and construct some steady state solutions. This suggests that this cannot serve as a model for the Chase Red Queen scenario. In Section \ref{s:pursuit}, we consider the case with aggregation $\beta>0$, but with $\alpha_H=0$, and start the construction of solutions having the form of two traveling pulses, the pathogen distribution tracking the escaping host distribution. The actual construction is achieved in Section \ref{s:proof-th}. It relies on perturbation techniques in rather intricate ad-hoc function spaces, on refined estimates for the eigenelements of the multivariate harmonic oscillator gathered in Appendix \ref{A:linear}, and a technical lemma on series involving binomial coefficients in Appendix \ref{B:proof-taupe}. This reveals that this acts as a model for the Chase Red Queen scenario without adding any external environmental force.  A short discussion is also presented in Section \ref{s:discussion}, between the setting of our main result, namely Theorem \ref{th:lin_pursuit}, and its proof in Section \ref{s:proof-th}, Appendix \ref{A:linear} and Appendix \ref{B:proof-taupe}.

\section{Numerical exploration of the possible outcomes}\label{s:num}

In this section, we will verify that model \eqref{sys:full} is capable of describing situations corresponding to the Chase Red Queen scenario, that is, situations where the mean phenotypes of the host $\xb(t)$ and the pathogen $\yb(t)$ do not converge but instead engage in a form of perpetual chase.

In all cases, we operate in two dimensions ($n=2$) and set $\ell=0$. The numerical resolution relies on a method of lines (a combination of a finite difference method for spatial discretization and the Runge-Kutta method for time integration), which easily handles non-local terms. The programs are written in Python, and are available in a \href{https://doi.org/10.17605/OSF.IO/W6FGV}{Jupyter notebook} and can also be executed on \href{https://colab.research.google.com/drive/1AJpZuQJEUQ8ESCWUBFFZCe71w0_CNgTx?usp=sharing}{Google Colab}. The parameter values are $\mu_H^2=\mu_P^2=0.1$, $R_H=4$, $R_P=1$, $\gamma_H=1$, $\gamma_P=0.01$, $\rho_{\max}=0.1$, $\theta=1$ and $\alpha_P=1$. The initial distributions of $h$ and $p$ have initial mass $10$ and are respectively concentrated at $\x_0=(0.5,0.5)$ and $\y_0=(0.7,0)$.

\paragraph{Model without aggregation term ($\beta=0$), Figure~\ref{fig:z}.}

We first consider the case $\alpha_H=0$ and $\beta=0$, meaning that, in the absence of pathogen, there is no phenotypic optimum for the host (all phenotypes have the same fitness). In this case, we observe that the host density appears to form a ring that diffuses towards infinity. Since the optimum for the pathogen is $\opt(\xb(t)) = \xb(t)$, the pathogen distribution tends to concentrate around $\xb(t)$ (which slightly deviates from its initial position). The hosts with the worst phenotype $\W(\yb(t)) = \yb(t)$ are therefore those whose phenotype is close to $\xb(t)$. Selection thus enables the host to avoid the pathogen by forming a ring-shaped distribution. This ring diffuses towards infinity by mutation, reducing the effect of the pathogen on the host, without diminishing the fitness of the pathogen.

When $\alpha_H>0$, the position $\mathbf{0}$ corresponds to a phenotypic optimum for the host in the absence of pathogen. We observe the same phenomenon as before, namely the formation of a ring. The essential difference from the previous case lies in the fact that $\xb(t)$ tends towards $\mathbf{0}$, and the ring tends to stabilize instead of diffusing to infinity, as hosts must find a compromise between being near $\mathbf{0}$ and away from the pathogens, which concentrate around $\xb(t)$, which itself tends towards $\mathbf{0}$.

Thus, under these assumptions, the model does not reproduce the Chase Red Queen scenario, and above all, it lacks realism: artificially, the pathogen maintains a high fitness by being close to $\xb(t)$, while the host distribution, in the form of a ring, avoids the position $\xb(t)$. The pathogen thus lives on hosts that do not even exist. In Section~\ref{s:static}, we offer an analysis of this model, which confirms that it does not capture the Chase Red Queen Scenario.

A more complex model, involving the entire distributions of $h$ and $p$ in the interactions between the two species, rather than just the mean, might not have this problem. However, its mathematical analysis is out of reach. Hence the introduction of the aggregation term with $\beta > 0$.

\begin{figure}
   \centering
   \begin{minipage}[c]{0.4\textwidth}
       \centering
       \includegraphics[scale=0.25]{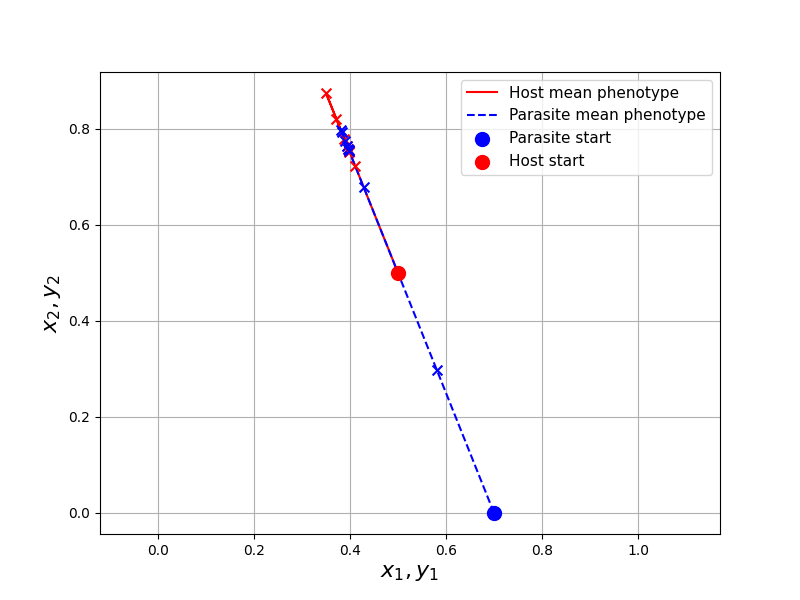} 
       \includegraphics[scale=0.25]{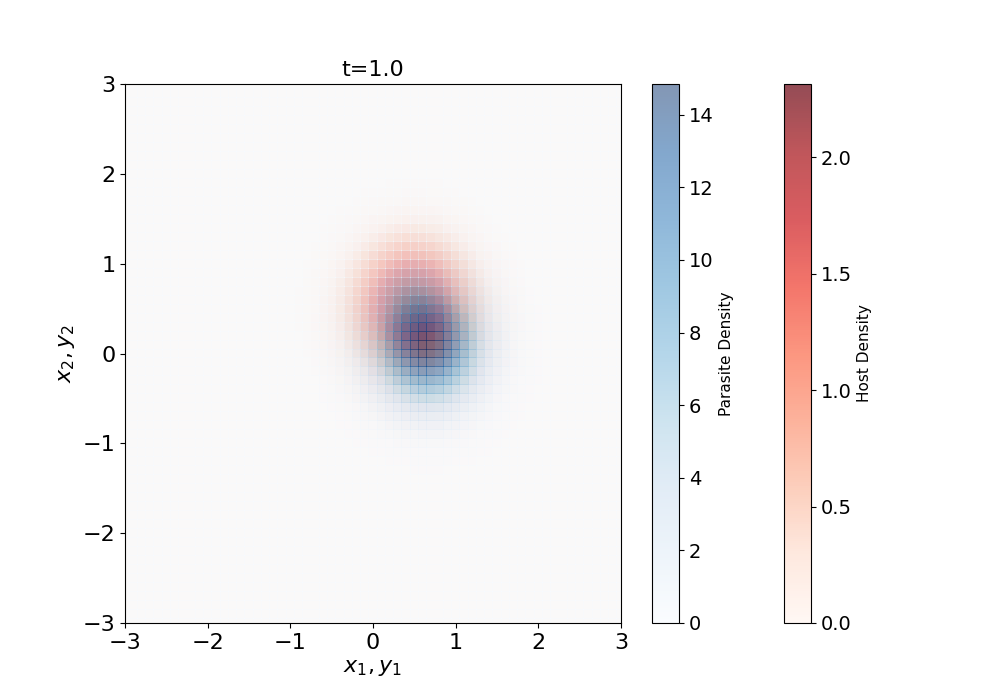}
       \includegraphics[scale=0.25]{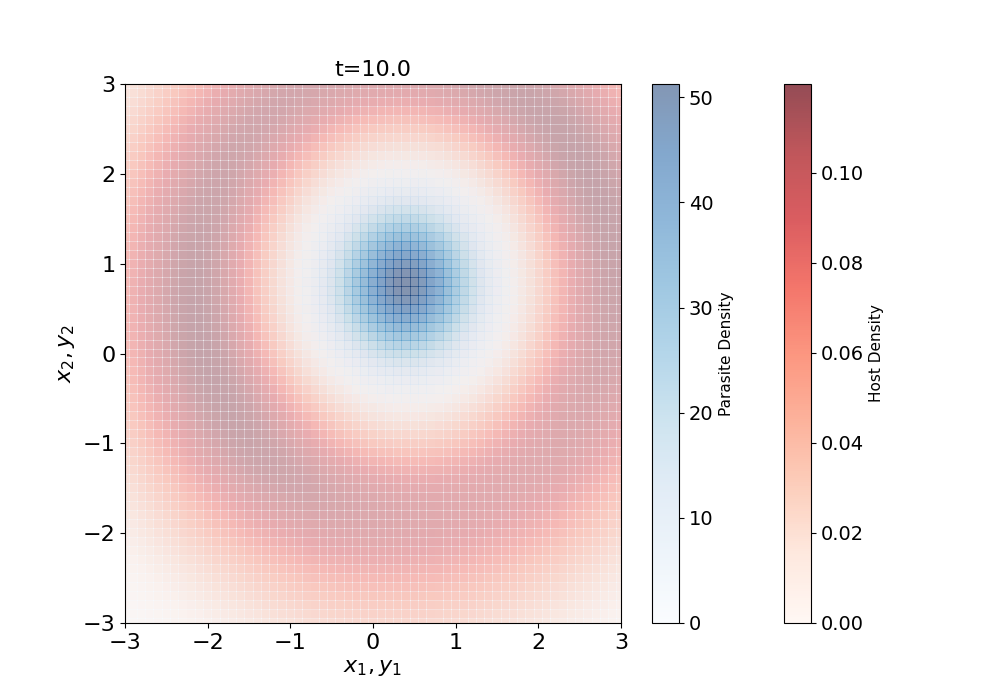}
       \includegraphics[scale=0.25]{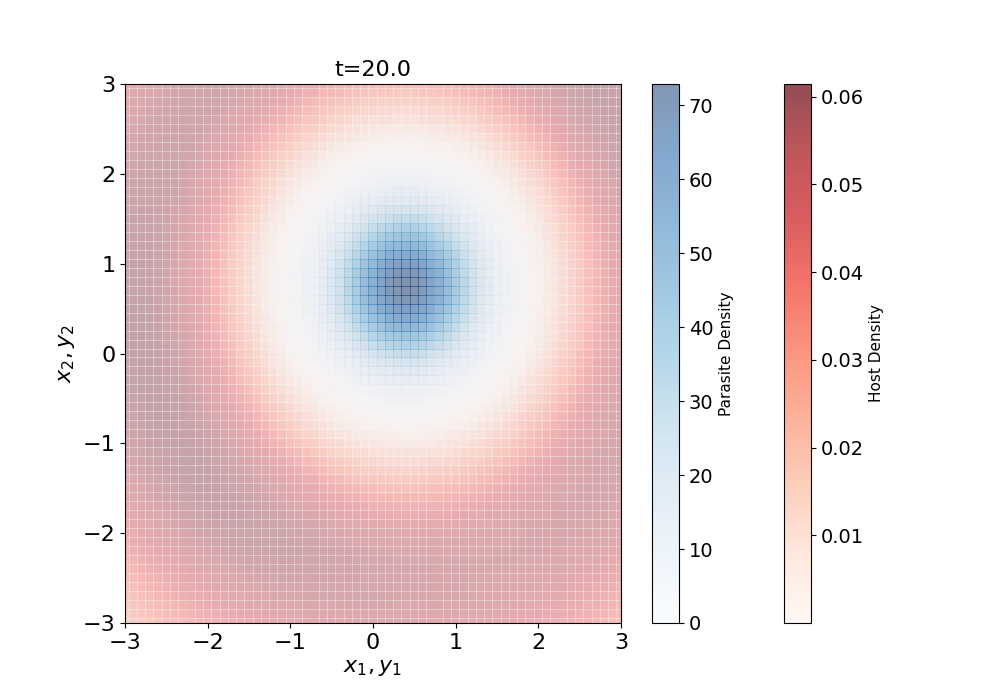}
   \end{minipage}
   \begin{minipage}[c]{0.4\textwidth}
       \centering
       \includegraphics[scale=0.25]{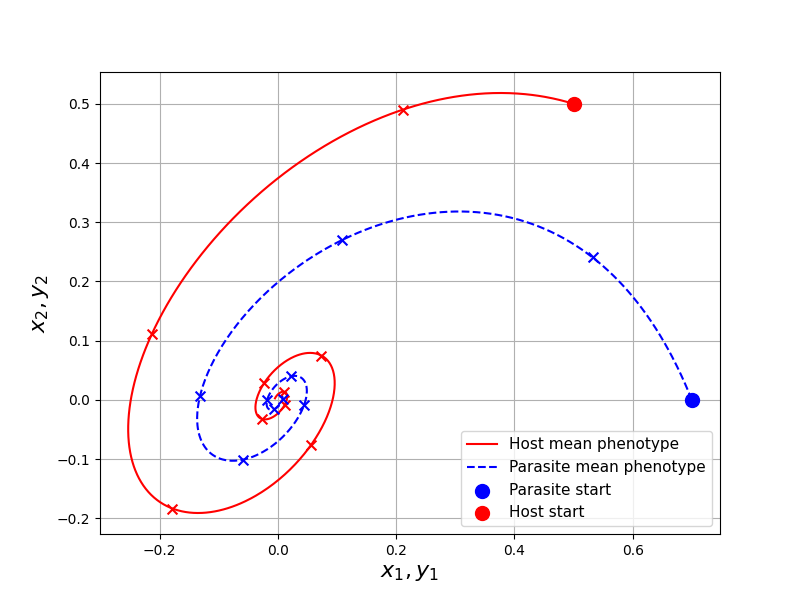}
       \includegraphics[scale=0.25]{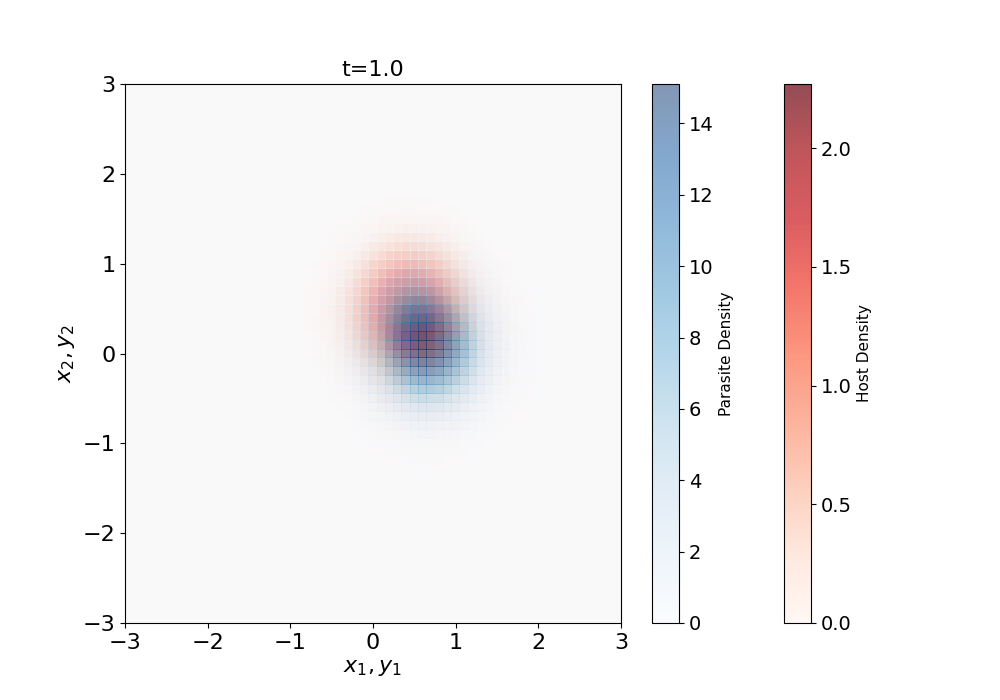}
       \includegraphics[scale=0.25]{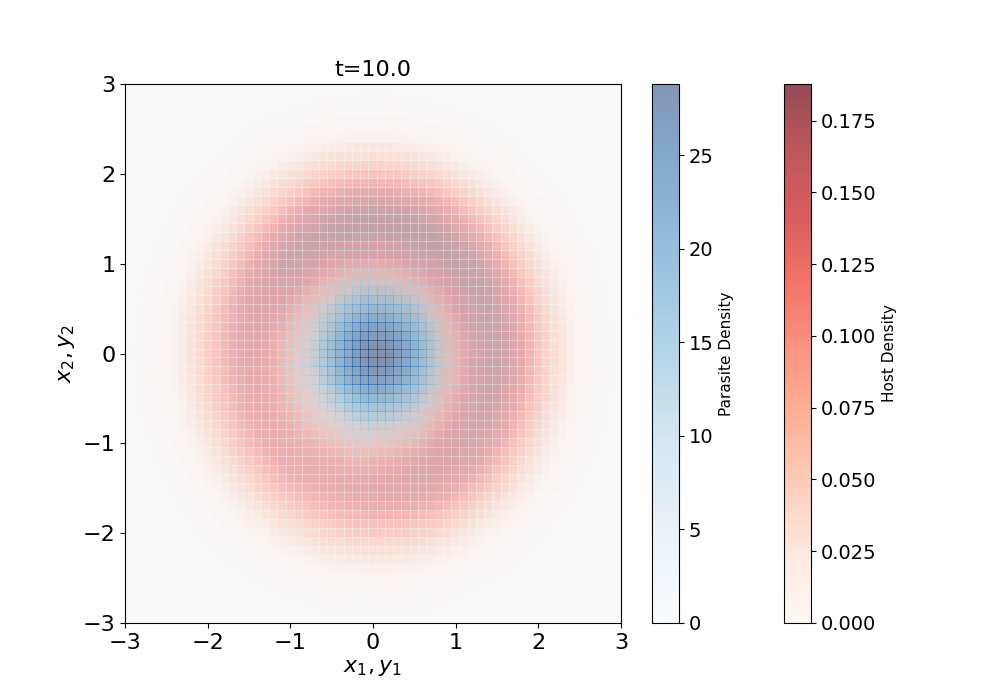}
       \includegraphics[scale=0.25]{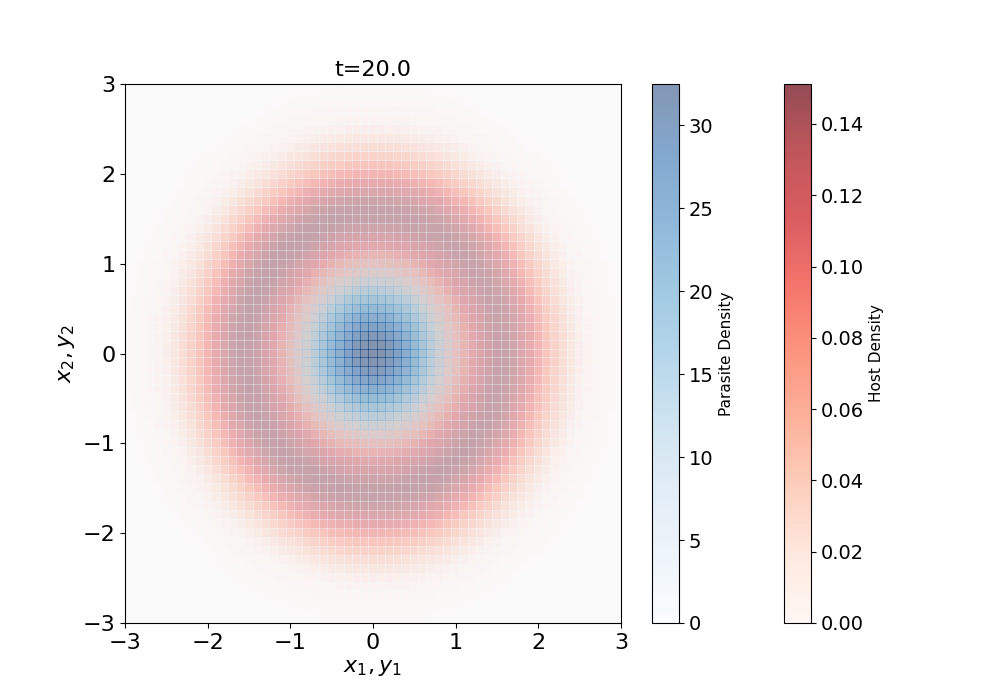}
   \end{minipage}
   \caption{\textbf{Model without aggregation term  ($\beta= 0$)}. Left column: model without phenotype optimum for the host ($\alpha_H=0$); Right column: model with a phenotype optimum $\x=0$ for the host ($\alpha_H=0.5$). The first row represents the trajectories of the mean phenotypes $\xb(t)$ and $\yb(t)$ for $t \in (0,20)$; the crosses correspond to the positions of the mean phenotypes at successive times, going by twos. The subsequent rows represent snapshots of the distributions $h(t,\x)$ and $p(t,\y)$ at successive times $t=1,\,10, \, 20.$ }
   \label{fig:z}
\end{figure}

\paragraph{Model with aggregation term ($\beta>0$), Figure~\ref{fig:n}.}
Once again, we begin by analyzing the case where $\alpha_H=0$ (no phenotypic optimum for the host). The mean phenotypes of the host and pathogen, $\xb(t)$ and $\yb(t)$, appear to move at a constant speed along a straight line defined by the initial phenotypes $\x_0$ and $\y_0$. The host density is bean-shaped, while the pathogen density remains Gaussian-shaped and follow the host density with a constant lag. These two densities seem to move at a constant speed along the line $(\x_0,\y_0)$  and with constant profiles, thus forming a couple of  traveling pulses. Thus, we have indeed achieved a scenario akin to the Chase Red Queen. Furthermore, this model also appears to show that a mobile optimum $\opt(t) = \opt_0 + ct \u$, as in \cite{AlfBer17,CalHen22,RoqPat20},  can naturally emerge from interactions between hosts and pathogens. We provide a rigorous analysis of this model and prove the existence of traveling pulses in Sections~\ref{s:pursuit} and \ref{s:proof-th}.

When $\alpha_H > 0$, the fitness of the host decays for phenotypes $\x$ away from the position $\mathbf{0}$. As a result, host densities can no longer shift to infinity to escape pathogens. We observe here that the mean phenotypes of the host and pathogen, $\xb(t)$ and $\yb(t)$, converge towards circular trajectories, rotating around $\mathbf{0}$ with a constant radial speed and a constant lag between the host and the pathogen. The host density is again bean-shaped, while the pathogen density is Gaussian-shaped. They form a pair of generalized traveling pulses, moving at constant speed along a circle. Thus, we again encounter a Chase Red Queen scenario, but this time with cyclic trajectories. In this case, from the pathogen's perspective, periodic trajectories for $\opt(t)$ emerge, as in \cite{CarNad20,FigMir19,LorChi15,RoqPat20},  but along a circle (these works only considered optima moving along a straight line). The mathematical analysis of the model under these assumptions will be the subject of a future study.

\begin{figure}
   \centering
   \begin{minipage}[c]{0.4\textwidth}
       \centering
       \includegraphics[scale=0.25]{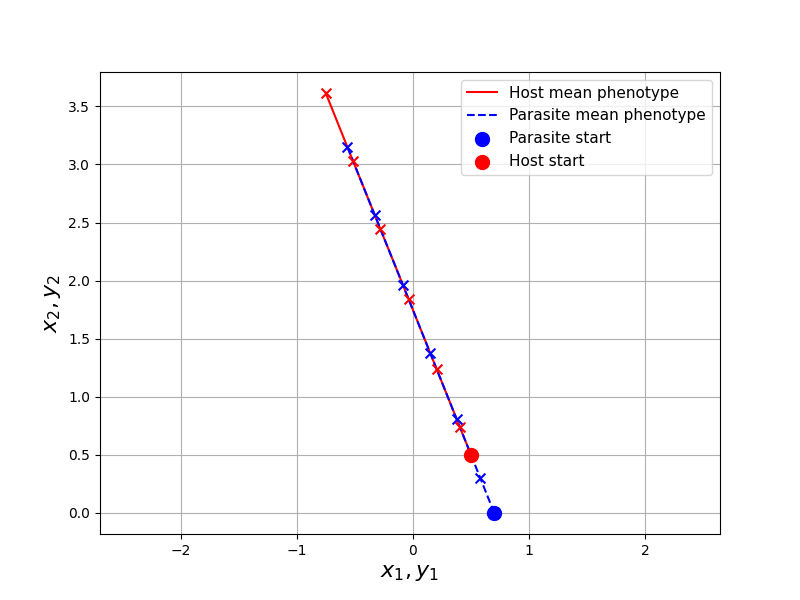}
       \includegraphics[scale=0.25]{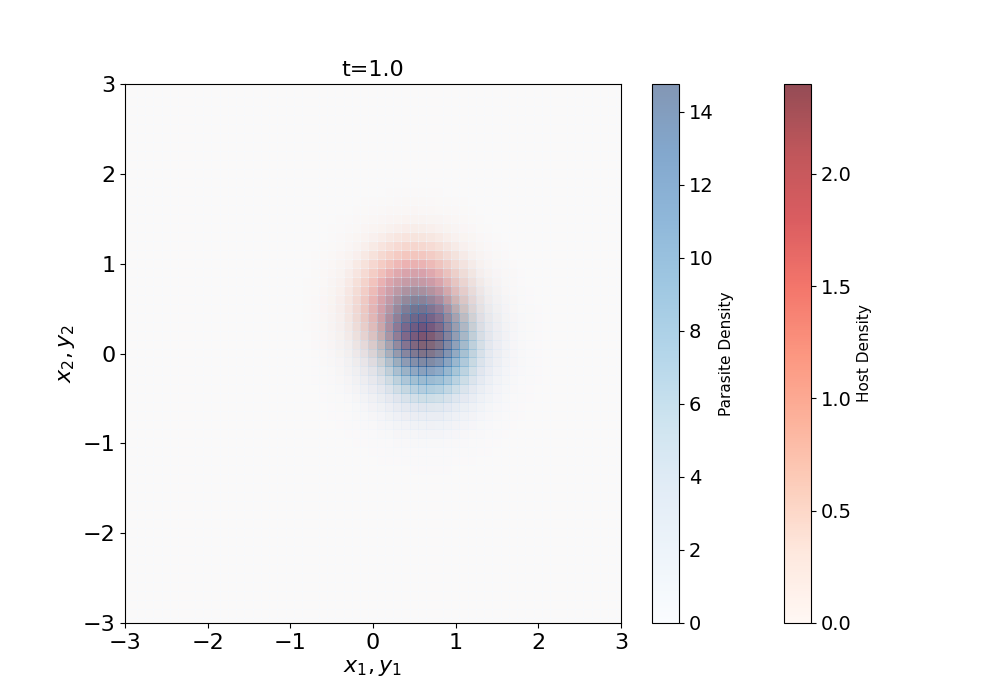}
       \includegraphics[scale=0.25]{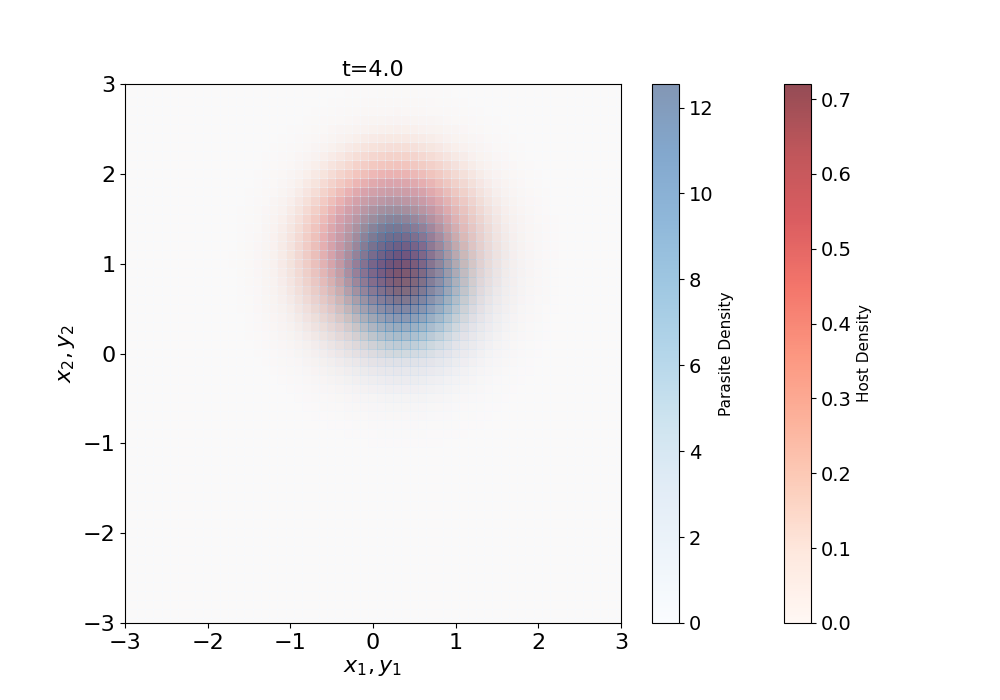}
       \includegraphics[scale=0.25]{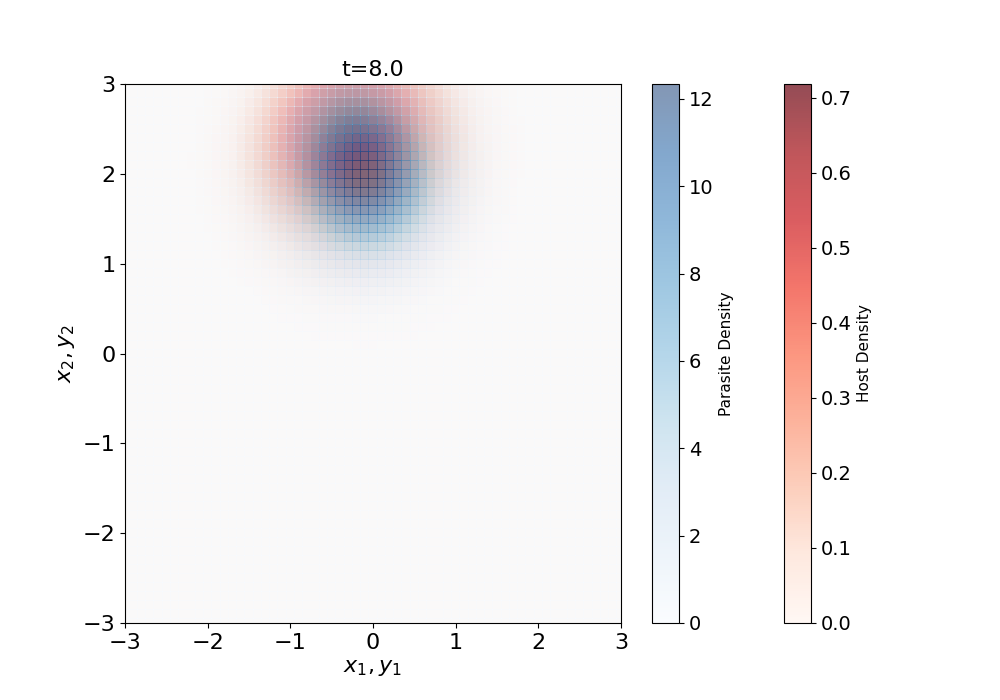}
   \end{minipage}
   \begin{minipage}[c]{0.4\textwidth}
       \centering
       \includegraphics[scale=0.25]{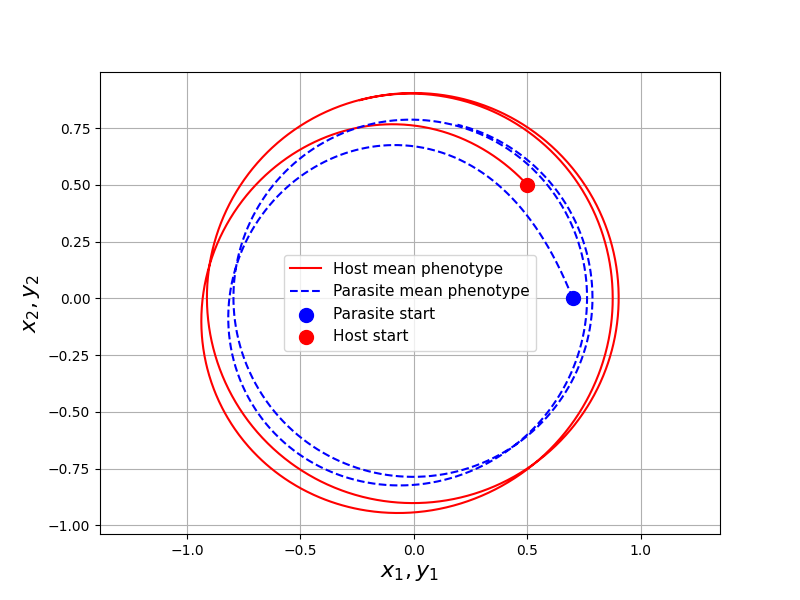}
       \includegraphics[scale=0.25]{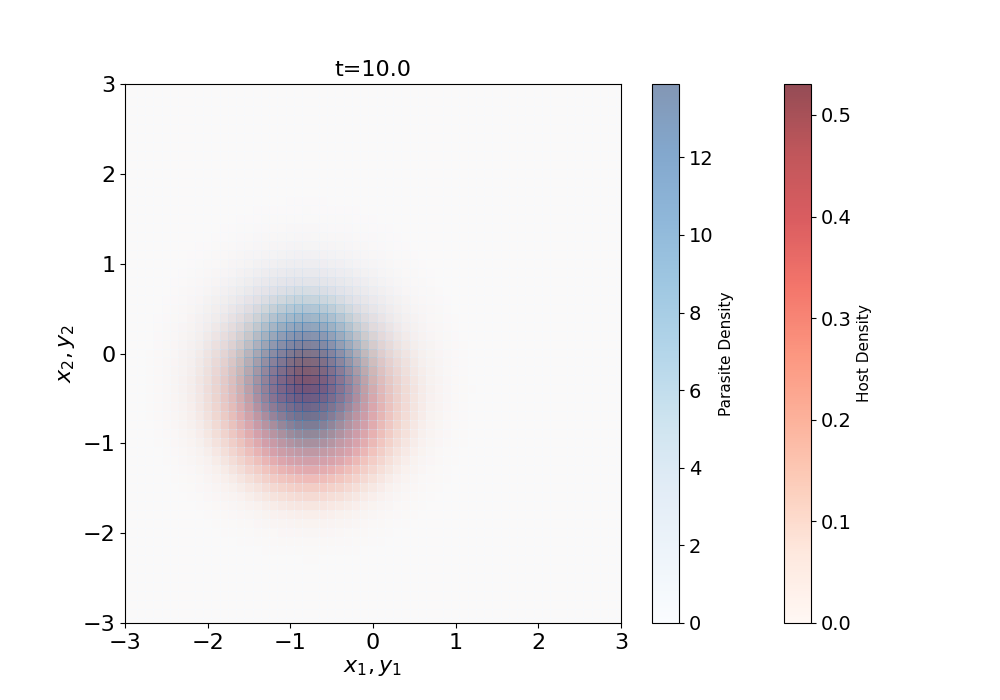}
       \includegraphics[scale=0.25]{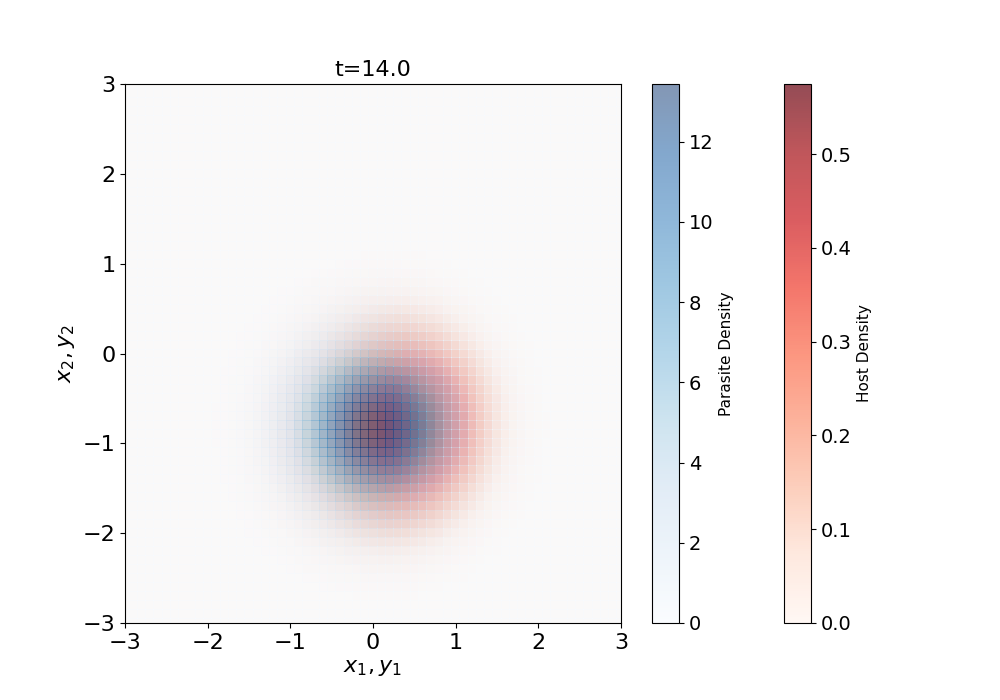}
       \includegraphics[scale=0.25]{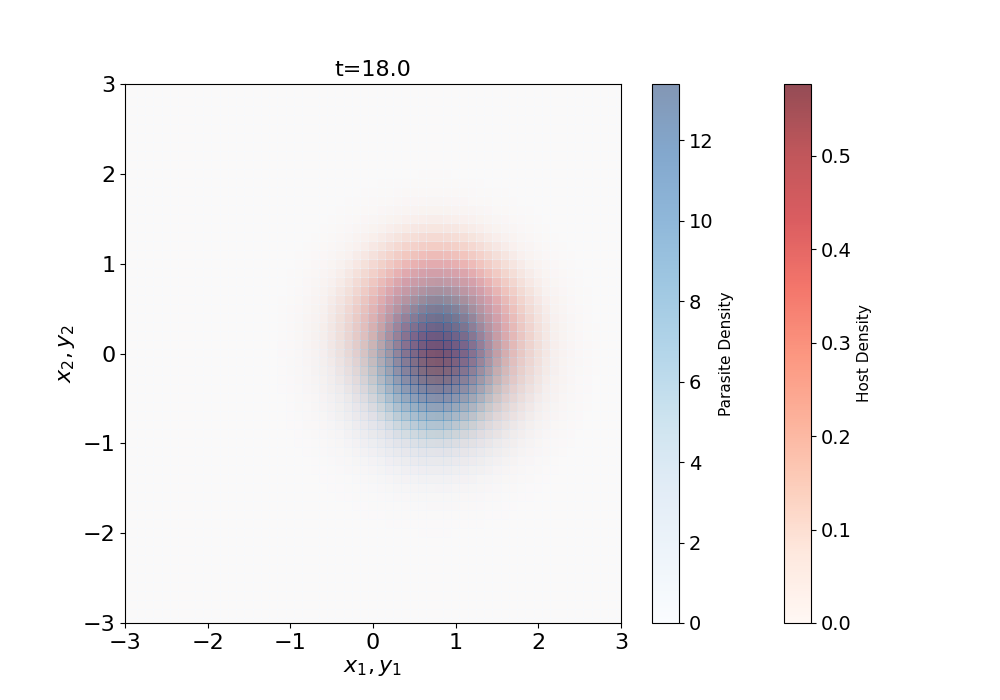}
   \end{minipage}
   \caption{\textbf{Model with aggregation term  ($\beta= 1$)}. Left column: model without phenotype optimum for the host ($\alpha_H=0$); Right column: model with a phenotype optimum $\x=0$ for the host ($\alpha_H=0.2$). The first row represents the trajectories of the mean phenotypes $\xb(t)$ and $\yb(t)$ for $t \in (0,12)$ (left) or $t \in (0,40)$ (right); the crosses correspond to the positions of the mean phenotypes at successive times, going by twos. The subsequent rows represent snapshots of the distributions $h(t,\x)$ and $p(t,\y)$ at different times.}
   \label{fig:n}
\end{figure}

\section{Model without aggregation term  \texorpdfstring{($\alpha_H>0$, $\beta=0$, $\ell=0$)}{}}\label{s:static}

In this section we focus on \eqref{sys:full} in the special case $\beta=0$.  If $\ell>0$, one would need to use a perturbation approach as performed in  the much more complex case of Section \ref{s:pursuit}. Since our goal is here to confirm the numerical evidence that $\beta=0$ is not adequate to capture the Chase Red Queen scenario, we assume $\ell=0$. We thus focus on
\begin{equation}  \label{sys:static}
    \left\{\begin{array}{ll}
        \ds \p_t h=\mu_H^2 \Delta_\x h +\left(R_H-\gamma_H H -\alpha_H^2 \|\x\|^2- P\rho_{\max} \, e^{-\, \theta \|\x-\yb(t)\|^2}\right) h  , & \quad t>0,\,  \mathbf x \in \R^n, \vspace{7pt} \\
        \ds \p_tp = \mu_P^2 \Delta_\y p + \left(R_P- \gamma_P \, \frac{P}{H}-\alpha_P^2 \|\y -\xb(t)\|^2\right) p, & \quad t>0,\,  \mathbf y\in \R^n,
    \end{array}\right.
\end{equation}
where all parameters are positive. In the sequel, we are concerned with the following particular solutions. 

\begin{defi}[Stationary solution]\label{def:stationary} A stationary solution of \eqref{sys:static} is a quadruplet $(H,P,\varphi,\psi)$ with $H>0$ (the host population size), $P>0$ (the pathogen population size), $\varphi$ and $\psi$ two positive probability densities on $\R^n$, that is
\begin{equation}
    \label{profiles-stat-1}
\varphi > 0,\quad  \psi > 0,\quad \int _{\R^n} \varphi(\x)\,\mathrm d\x=\int _{\R^n} \psi(\y)\,\mathrm d\y=1,
\end{equation}
with zero means
\begin{equation}
    \label{profiles-stat-2}
 \int _{\R^n} \x\varphi(\x)\,\mathrm d\x=\int _{\R^n} \y\psi(\y)\,\mathrm d\y=\0 ,
\end{equation}
such that the couple $ (h(t,\x),p(t,\y))=(H\varphi(\x),P\psi(\y))$  solves \eqref{sys:static}.
\end{defi}

\begin{theorem}[Stationary state]\label{th:stat}
A stationary solution of \eqref{sys:static} exists if and only if
    \begin{equation}\label{cond-stationary}
        R_P>n\mu_P\alpha_P \qquad \hbox{ and } \qquad R_H>n\mu_H\alpha_H.
    \end{equation}
    If \eqref{cond-stationary} holds, the stationary state is unique, radial, $\psi$ is explicit
    \begin{equation}\label{psi-explicit}
  \psi(\y)=\left(\frac{\alpha_P}{2 \pi\mu_P} \right) ^{n/2}\exp\left(-\, \frac{\alpha_P}{2 \mu_P}\, \|\y \|^2\right),
   \end{equation}
    and we have the relation $R_P-\gamma_P \, \frac{P}{H}=n\mu_P\alpha_P$. 
\end{theorem}

\begin{proof} We plug the ansatz $(h(t,\x),p(t,\y))=(H\varphi(\x),P\psi(\y))$ into \eqref{sys:static}. Observe that  \eqref{profiles-stat-2} enforces
$$
\xb(t)=\0 \quad \text{and }\quad  \yb(t)=\0,
$$
so that we get the elliptic system (where $H$ and $P$ are also to be determined)
\begin{equation}  \label{sol-stat}
\left\{\begin{array}{ll}
    \ds \mu_H^2 \Delta \varphi+\left( R_H-\gamma_H H -\alpha_H^2 \Vert \x\Vert ^2 -P \rho_{\max} e^{-\theta\Vert \x\Vert ^2}\right)\varphi=0, & \quad \x \in \R^n, \vspace{7pt}
    \\
  \ds \mu_P^2 \,\Delta \psi+ \left(R_P- \gamma_P \, \frac{P}{H}-\alpha_P^2 \|\y\|^2\right) \, \psi  =0,  & \quad  \y \in \R^n.
     \end{array}\right.
\end{equation}

On the one hand, from the second equation in \eqref{sol-stat}, the function $\psi$ and the scalar $R_P- \gamma_P \, \frac{P}{H}$ have to be the ($L^1$ normalized) principal eigenelements of  the differential (harmonic oscillator) operator 
$-\mu_P^2 \,\Delta +\alpha_P^2\|\y\|^2$, that is (see Proposition \ref{prop:basis_eigenfunctions}) $\psi$ is given by \eqref{psi-explicit} and
\begin{equation}
\label{H(P)}
R_P-\gamma_P \frac P H=n\mu_P \alpha_P\qquad \Leftrightarrow \qquad H=H(P):= \frac{\gamma_P}{ R_P-n\mu_P\alpha_P}P,
\end{equation}
which, in particular, enforces $R_P>n\mu_P\alpha_P$. 

On the other hand,  for  fixed $P\ge 0$, we define $(\lambda_{P},\varphi_{P})$ as the unique principal eigenpair such that $\varphi_{P}>0$ in $\R^n$,
\begin{equation}\label{ODE-tphi}
\mu_H^2 \Delta \varphi _{P}+\left(R_H-\gamma_H H(P) -\alpha_H^2 \Vert \x\Vert ^2-P \rho_{\max} e^{-\theta\Vert \x \Vert ^2}\right)\varphi _{P}=\lambda _{P}\, \varphi _{P}, 
\end{equation}
with the normalisation condition $
\int_{\R^n}  \varphi _{P} (\x) \, \mathrm d\x =1$. From the invariance by rotation of the operator, we know that $\varphi_P$ is radial, and so its mean value is null. We know that $\varphi  _{P} \in H^1(\R^n)\cap L^2_{w}(\R^n)$, where
\[
L^2_w(\R^n):=\left\{f:\R^n\to\R\hbox{ such that } \x \mapsto \Vert \x \Vert\,f(\x) \in L^2(\R^n)\right\},
\]
and that the Rayleigh formula is available, namely
\begin{equation}\label{Rayleigh2}
\lambda  _{P}=R_H -\gamma_H H(P)-\min \left\{ Q _{P}(f): f\in H^1(\R^n)\cap L^2_{w}(\R^n), \int_{\R^n} f^2(\x)\,\mathrm d\x=1\right \},
\end{equation}
where
\begin{equation}\label{eq:rayleigh_Q}
Q _{P}(f):=\mu_H^2\int_{\R^n} \Vert \nabla f (\x)\Vert ^2 \, \mathrm d\x+\int_{\R^n} \left(\alpha_H^2 \Vert \x\Vert ^2+P\rho_{\max} e^{-\theta\Vert \x\Vert ^2}\right)f^2(\x)\, \mathrm d\x.
\end{equation}
The function $P\mapsto \lambda_P$ is continuous and, as clear from \eqref{H(P)}, \eqref{Rayleigh2} and \eqref{eq:rayleigh_Q}, decreasing on $[0,+\infty)$.

When $P=0$, the eigenpair is
\[
 \varphi_{0}(\x)=\left(\alpha_H\over 2 \,\pi\mu_H\right) ^{n/2} \exp \left(-\, \frac{\alpha_H}{2\, \mu_H}\, \|\x\|^2\right), \quad \lambda_0=R_H-n\mu_H\alpha_H.
\]
Moreover, we have that
\begin{eqnarray*}
Q _{P}(f)&\geq & \mu_H^2\int_{\R^n} \Vert \nabla f (\x)\Vert ^2 \,\mathrm d\x+\int_{\R^n} \alpha_H^2\Vert \x\Vert ^2 f^2(\x)\,\mathrm d\x,\\
&\geq & \min\left \{  \mu_H^2\int_{\R^n} \Vert \nabla f (\x)\Vert ^2 \,\mathrm d\x+\int_{\R^n} \alpha_H^2 \Vert \x\Vert ^2 f^2(\x)\, \mathrm d\x\right\}=n\mu_H \alpha_H,
\end{eqnarray*}
implying that
\[
\lambda _{P}\leq R_H-n\mu_H \alpha_H-\gamma _H H(P) \to -\infty \quad \text{ as } P\to +\infty.
\]
As a result, there is (a unique) $P>0$ such that $\lambda_P=0$ (which corresponds to constructing a stationary state)  if and only if $\lambda_0>0$, from which the result follows. 
\end{proof}

\section{Pursuit model  \texorpdfstring{($\alpha_H=0$, $\beta>0$, $\ell>0$)}{}}
\label{s:pursuit}

The objective of the rest of the present work  is to prove the existence of traveling pulses moving along a straight line, as observed in Section~\ref{s:num} and Figure~\ref{fig:n} (left column), under the same assumptions that is $\alpha_H=0$ and $\beta>0$. We start with the problem where the pathogen has no impact on the host ($\rho_{max}=0$) and show that in this case, there exist stationary pulses with velocity $c=0$. By perturbing this parameter, we manage to construct pulses moving at a velocity $c>0$. This perturbation technique only works when $\ell >0$ and not in the case $\ell=0$.

 Indeed, in the case $\ell=0$, the stationary solutions ($\rho_{max}=0$) we construct are symmetric with respect to $\mathbf{0}$; after perturbation, due to the symmetry of the problem, the existence of a solution with $c>0$ would imply the existence of a solution with $c<0$. However, the constructed solution is unique and would therefore also have zero velocity. Numerical computations (not shown) indicate that this solution corresponds to a ring-shaped host density, as in the case $\beta=0$ and is unstable: we obtain it as the long time behavior of the Cauchy problem only when $\x_0=\y_0$, i.e. when the problem is symmetric with respect to $\x_0=\y_0$. On the other hand, as soon as $\x_0\neq \y_0$ the numerical results of Section~\ref{s:num} show that traveling pulses can  be achieved even with $\ell=0$.

 In light of this, to rigorously establish the existence of traveling pulses with nonzero speed, we  break the symmetry of the problem by assuming the technical condition $\ell>0$. Our main result is presented in Theorem~\ref{th:lin_pursuit}.

\medskip

We thus  focus on \eqref{sys:full} in the special case $\alpha_H=0$, $\beta>0$, $\ell>0$, that is 
\begin{equation}  \label{sys-run}
    \left\{\begin{array}{ll}
        \ds \p_t h=\mu_H^2 \Delta_\x h +\left(R_H-\gamma_H H -\beta^2\|\x-\xb(t)\|^2- P\rho_{\max} \, e^{-\, \theta \|\x-\yb(t)\|^2} \right)  h, & \quad t>0,\,  \mathbf x \in \R^n,\vspace{7pt} \\
        \ds \p_tp = \mu_P^2 \Delta_\y p + \left(R_P- \gamma_P \, \frac{P}{H}-\alpha_P^2 \|\y+\ell\u -\xb(t)\|^2\right)p, & \quad t>0,\,  \mathbf y \in \R^n,
    \end{array}\right.
\end{equation}
where all parameters are positive, and $\u$ is a given unit vector. In the sequel, we are concerned with the following particular solutions.

\begin{defi}[Pursuit pulse]\label{def:pulse}  A pursuit pulse solution of \eqref{sys-run} is a sextuplet $(c,\tau,H,P,\varphi,\psi)$ with $c\neq  0$ (the propagation speed), $\tau>0$ (the delay of the pathogen), $H>0$ (the host population size), $P>0$ (the pathogen population size), $\varphi$ and $\psi$ two  positive probability densities  on $\R^n$, that is
\begin{equation}
    \label{profiles-gen}
\varphi > 0,\quad  \psi > 0,\quad \int _{\R^n} \varphi(\z)\,\mathrm d\z=\int _{\R^n} \psi(\w)\,\mathrm d\w =1,
\end{equation}
with means
\begin{equation}
    \label{means}
 \int _{\R^n} \z\varphi(\z)\,\mathrm d\z=\0, \quad \int _{\R^n} \w\psi(\w)\,\mathrm d\w=-\ell \u,
\end{equation}
such that the couple
\begin{equation}
 (h(t,\x),p(t,\y)) := \left(H \varphi(\x- ct\u),P \psi(\y-c(t-\tau)\u)\right),
        \label{ansatz:lin}
\end{equation}
solves \eqref{sys-run}.
\end{defi}

We now plug the ansatz \eqref{ansatz:lin} into \eqref{sys-run}. Observe that  \eqref{profiles-gen} and \eqref{means}  enforce
\begin{equation}
    \xb(t)=c  t  \u \quad \text{ and } \quad  \yb(t)= \left(c(t-\tau)-\ell\right)\u,
\label{ansatz:speed}
\end{equation}
so that we get the elliptic system (where $c$, $\tau$, $H$ and $P$ are also to be determined)
\begin{equation} \label{eq:tw-2}
    \left\{
    \begin{array}{lll}
& -c\nabla \varphi \cdot \u  = \mu_H^2 \Delta \varphi + \left(R_H-\gamma_H H -\beta^2 \Vert \z \Vert ^2-P\rho_{\max} e^{-\theta\Vert \z+(c\tau+\ell) \u \Vert ^2 }\right) \varphi  ,  & \qquad \z\in \R^n, \vspace{7pt}
\\
 & -c\nabla \psi\cdot\u  = \mu_P^2 \Delta \psi + \left(R_P-\gamma_P \ds \frac P H-\alpha_P^2\Vert \w -(c\tau-\ell)\u \Vert ^2\right)\psi,  & \qquad \w \in \R^n, 
    \end{array}
    \right.
\end{equation}
for $\varphi=\varphi(\z)$, $\psi=\psi(\w)$ where, roughly speaking, $\z=\x-ct\u$ and $\w=\y-c(t-\tau)\u$. 

\medskip

\noindent {\bf Determination of $\psi$.}  Letting
\begin{equation}
\overline \psi(\w) := \exp\left(\frac c{2\mu_P^2}\, \w\cdot \u\right)\, \psi(\w),
\end{equation}
the second equation in \eqref{eq:tw-2} becomes
\[
 \mu_P^2 \Delta \overline \psi+\left(R_P-\gamma_P \, {P\over H} -\, {c^2\over 4\mu_P^2} -\alpha_P^2 \|\w-(c\tau-\ell)\u \|^2\right)\overline \psi=0.
\]
Thus, the function $\overline \psi$ and the scalar $R_P-\gamma_P \, {P\over H} -\, {c^2\over 4\mu_P^2}$ have to be the principal eigenelements of  the differential (harmonic oscillator) operator 
$-\mu_P^2 \,\Delta +\alpha_P^2\|\w -(c\tau-\ell) \u\|^2$, that is (see Proposition \ref{prop:basis_eigenfunctions})
\[
\overline \psi(\w)=K \exp \left(-\, {\alpha_P\over 2\, \mu_P}\, \|\w-(c\tau-\ell) \u\|^2\right) \qquad\hbox{ and } \qquad R_P-\gamma_P {P\over H} -{c^2\over 4\mu_P ^2}=n\mu_P \alpha_P,
\]
for some normalization constant $K>0$ ensuring $\int_{\R^n}\psi(\w)\,\mathrm d\w=1$.  From now, we assume without loss of generality that
\begin{equation}\label{choix-u}
\u=(1,0,\dots,0),
\end{equation}
and check, using straightforward computations based on Lemma \ref{lem:gaussiennes}, that 
\begin{equation}
\label{K}
K^{-1}=\left(\frac{2\pi \mu_P}{\alpha_P}\right)^{\frac n2} e^{-\frac{(c\tau-\ell)c}{2\mu_P^2}+\frac{c^2}{8\alpha_P\mu_P^3}}.
\end{equation}

\begin{lem}\label{lem:gaussiennes} For any $a>0$, $b\in \R$, $c\in \R$, one has
$$
\int_\R e^{-ax^2+bx+c}\,\mathrm dx=\sqrt{\frac{\pi}{a}} \, e^{\frac{b^2}{4a}+c}, \qquad 
\int_\R x e^{-ax^2+bx+c}\,\mathrm dx=  \frac{b}{2a} \sqrt{\frac{\pi}{a}} \, e^{\frac{b^2}{4a}+c}.
$$
\end{lem}

 As for the mean value of $\psi$, we need to ensure $\int_{\R^n} w_1\psi(\w)\,\mathrm d\w=-\ell$. Using again straightforward computations based on Lemma \ref{lem:gaussiennes}, we see that this requires the relation
 $$
 c\left(\tau-\frac{1}{2\mu_P\alpha_P}\right)=0
 $$ 
 to hold.
 
To sum up, we have the following.

\begin{lem} \label{lem:psi}  If a pursuit pulse solution $(c,\tau,H,P,\varphi,\psi)$ of \eqref{sys-run} exists then, necessarily, 
\begin{equation}
\label{1}
\psi(\w)=K \exp\left(-\frac c{2\mu_P^2}\, \w\cdot \u\right)\ \exp \left(-\, {\alpha_P\over 2\, \mu_P}\, \|\w-c\tau \u\|^2\right),
\end{equation}
where $K>0$ is given by \eqref{K} (so that $\int _{\R^n}\psi(\w)\, \mathrm d \w=1$),
\begin{equation}
\label{2}
\tau=\frac{1}{2\mu_P\alpha_P}
\end{equation}
(so that $\int_{\R^n} \w \psi(\w)\,\mathrm d\w=-\ell \u$), and we have the relation
\begin{equation}
\label{3}
R_P-\gamma_P {P\over H} -{c^2\over 4\mu_P ^2}=n\mu_P \alpha_P.
\end{equation}
\end{lem}

Therefore it remains to find the function $\varphi$, the speed $c$ (which has to be nonzero) and the population size $H$. Obviously, the difficulty comes from the exponential term in the equation for $\varphi$. Since the case $\rho_{\max}=0$ is decoupled, our strategy consists in using $\rho_{\max}$ as a perturbation parameter to finish the construction of pursuit pulse when $0<\rho_{\max}\ll 1$. 

When $\rho_{max}=0$ the pulse is actually stationary, as revealed by the  following which  is proved using exactly the same arguments as above for $\psi$.

\begin{lem}\label{lem:rhomaxzero} Assume
\[
R_H>n\mu_H\beta.
\]
If $c^0\in\R$, $H^0\in \R$ and $\varphi^0$ a probability density on $\R^n$ with zero mean are such that
\begin{equation} -c^0\nabla \varphi^0 \cdot \u  = \mu_H^2 \Delta \varphi^0 + \left(R_H-\gamma_H H^0 -\beta^2 \Vert \z \Vert ^2\right) \varphi ^0 ,   \qquad \z\in \R^n, 
\end{equation}
then they are given by
\[
c^0=0, \qquad   \varphi^0(\z)=\displaystyle \left({\beta\over 2\pi \mu_H}\right)^{n/2} \exp\left(-\, {\beta \over 2\mu_H}\, \|\z\|^2 \right), \qquad     H^0 =\displaystyle {R_H - n\mu_H \beta \over \gamma_H}>0.
\]
\end{lem}

The main result of this work is the following.

\begin{theorem}[Run straight for your life]\label{th:lin_pursuit} Let $n=1$ or $n=2$. Let $\mu_H$, $R_H$, $\gamma_H$, $\beta$, $\theta$, $\mu_P$, $R_P$, $\gamma_P$, $\alpha_P$  be positive parameters such that
$$
 R_P>n\mu_P\alpha_P, \quad R_H > n\mu_H \beta ,\quad \beta >\max\left( 3\times2^{n-2}-1,\frac 1{\sqrt 5-2}\right)\mu_H\theta.
$$
Then there is $\ell_0>0$ small enough such that,  for all $0<\ell<\ell_0$, the following holds.

There exists $\ep^*>0$ small enough such that, for all $0<\rho_{\max}<\ep^*$, there is a pursuit pulse solution $(c,\tau,H,P,\varphi,\psi)$ to \eqref{sys-run}. Furthermore, $\psi$ is given by \eqref{1}, the delay $\tau$ by \eqref{2}, we have the relation \eqref{3}, $c>0$ and $H<H^0={R_H - n\mu_H \beta \over \gamma_H}$.
\end{theorem}

 Our result holds in dimension 1 or 2. However, it would suffice to improve the technical Lemma \ref{lem:taupe-plus} for the construction to hold in any dimension, see Remark \ref{rem:num-conj} for further details. Note also that $H<H^0$ means that the pathogen has  a negative impact on the host population size.

\section{Discussion}\label{s:discussion}

Inspired by recent literature on phenotypically-structured PDE models, particularly those in which the phenotypic optimum is mobile, we developed a new host-pathogen coevolution model. In this model, the pathogen follows a classic adaptation dynamic in the presence of a mobile optimum; however, unlike existing PDE models, the trajectory of the optimum here results from interactions with a host population. Our results, derived from rigorous mathematical analysis and numerical simulations, illustrate scenarios where hosts and pathogens engage in a perpetual evolutionary race, characterized by the formation of traveling pulses and cyclical adaptations.

Although the model is based on simplifying assumptions, it provides an initial mathematical justification for the emergence of moving optima, such as those considered in \cite{AlfBer17,CalHen22} (constant speed), or \cite{CarNad20,FigMir19,LorChi15,RoqPat20} (periodic fluctuations), without relying on external forcing terms (e.g., climate change, pharmacokinetics, etc.). Furthermore, in all these studies (except \cite{Lav23}), the assumption was that the moving optimum traveled either in a one-dimensional space or along a straight line in a higher-dimensional space. Simulations in the case where $\alpha, \, \beta>0$ (Figure~\ref{fig:n}, right column), which include an aggregation term and a phenotype optimum for the host, show that curvilinear trajectories (in this case, circular) can also naturally emerge.

Among these simplifying assumptions, we assumed that the interactions terms between hosts and pathogens are influenced solely by the mean phenotypes and the total populations. This simplification makes the model more tractable for mathematical analysis. However, as a consequence of this assumption, unrealistic behaviors can emerge (see Figure~\ref{fig:z}). To mitigate these, we introduced an aggregation term for the host population (case $\beta>0$), which forces the phenotypic population to remain concentrated around its average value. This aggregation term is crucial for the emergence of Chase Red Queen scenarios.

Among the future research directions, we can mention the improvement of certain technical results (see Remark \ref{rem:num-conj}) that would allow us to extend our findings to dimensions $n \geq 3$. Future research should also focus on rigorously demonstrating the existence of periodic circular trajectories, as observed numerically. Another promising direction could involve introducing non-perturbative methods to prove results over a broader range of parameters. If the pathogen's distribution remains Gaussian, as in the existing literature (e.g., \cite{HamLav20}), numerical simulations suggest that the pathogen's effect on the host leads to a bean-shaped distribution of the latter. Therefore, we cannot expect to explicitly construct Gaussian solutions to the system. One possible approach could be to use methods based on the principal eigenfunctions of elliptic operators, as discussed in Section~\ref{s:static}. Finally, the emergence of periodic trajectories along circles opens the door to revisiting the existing results in \cite{CarNad20,FigMir19,LorChi15,RoqPat20} for more general periodic trajectories.

\section{Proof of Theorem \ref{th:lin_pursuit}}\label{s:proof-th}

 Let us recall that it remains to construct $c\neq 0$, $H>0$ and $\varphi$ a probability density with zero mean so that the first equation in \eqref{eq:tw-2} holds.

Let us remark that, recalling \eqref{choix-u}, if we are equipped with a probability density $\varphi$ solving the first equation in \eqref{eq:tw-2} and satisfying
\begin{equation}
    \int_{\R^n} z_1\varphi(\z)\,\mathrm d\z=0,
    \label{hyp:x1}
\end{equation}
then the function $\tilde \varphi$ defined by
\[
 \tilde \varphi(\z) := \frac {\varphi(\z)+\varphi(\iota(\z))} 2 , \quad \hbox{ where } \quad \iota(\z):=(z_1,-z_{2},-z_3,\dots,-z_n),
\]
is still a probability density   solving the first equation in \eqref{eq:tw-2} and, additionally, it has zero mean. Thus, it is enough to consider the constraint \eqref{hyp:x1}.

In the sequel use $\ep$ as a shortcut for $\rho_{\max}$. As explained above, we will rely on perturbations technics around the case $\ep=0$ solved in Lemma \ref{lem:rhomaxzero}. We recall the {\it Implicit Function Theorem}, see \cite[Theorem 4.B]{Zei-86} for instance.

\begin{theorem}[Implicit Function Theorem]\label{thm:implicit-functions}
	Let $X$, $Y$ and $Z$ be three Banach spaces. Suppose that
	\begin{enumerate}
		\item [(i)] The mapping $\mathcal F:U\subset X\times Y\to Z$ is defined on an open neighborhood $U$ of $(x_0, y_0)\in X\times Y$ and $\mathcal F(x_0, y_0)=0$. 
		\item [(ii)] The partial Fr\'echet derivative of $\mathcal F$ with respect to $y$ exists on $U$ and
			$$
			\mathcal F_y(x_0, y_0):Y\to Z \text{  is bijective}.
			$$
		\item [(iii)] $\mathcal F$ and $\mathcal F_y$ are continuous at $(x_0,y_0)$. 
	\end{enumerate}
	Then, the following properties hold.
	\begin{enumerate}
		\item [(i)]  {\rm Existence and uniqueness}. There exist $r_0>0 $ and $r>0$ such that,  for every $x\in X$ satisfying $\Vert x-x_0\Vert\leq r_0$, there exists a unique $y(x)\in Y$ such that $\Vert y-y_0\Vert\leq r$ and $\mathcal F(x, y(x))=0$.
		\item [(ii)] {\rm Continuity}. If $\mathcal F$ is continuous in a neighborhood of $(x_0,y_0)$, then the mapping  $x\mapsto y(x)$ is continuous in a neighborhood of $x_0$. 
		\item [(iii)] {\rm Higher regularity}. If $\mathcal F$ is of the class $C^m$, $1\leq m\leq \infty$,  on a neighborhood of $(x_0, y_0)$, then $x\mapsto y(x)$ is also of the class $C^m$ on a neighborhood of $x_0$.
	\end{enumerate}
\end{theorem}

In the sequel, we shall use the notations (multi-indexes $\k$ or $\j$ in $\N^n$, their $L^1$ norm $\sigma(\k)$, $\sigma(\j)$, eigenfunctions $\Gamma _\k=\Gamma_\k(\z)$, etc.) and results of Appendix \ref{A:linear}.

\subsection{The functional spaces and the map \texorpdfstring{$\mathcal F$}{}}


For $b>0$ to be precised later, we consider the spaces
\begin{equation*}
\mathcal X:=\left\{ (c,\varphi)\in \R\times  C^{2}(\mathbb{R}^{n})\left|\begin{array}{c}
  \exists C>0, \forall  \j \in \N ^n \text{ with } \sigma(\j) \leq 2, \forall \z \in \R^n, \vspace{5pt}\\
  \vert D^{\j }\varphi (\z)\vert \leq C\ds\prod_{i=1}^n{1\over (1+z_i^2)^2},\vspace{10pt}\\
\text{ and }\vspace{5pt}\\
  \exists K>0, \forall  {\j} \in \N ^n \text{ with } \sigma(\j) \leq 2, \forall \k \in \N^n,\vspace{5pt}\\
\ds \left|\int_{\mathbb{R}^n}D^{\j}\varphi(\z)\Gamma_{\k}(\z)\,\mathrm d\z\right|\leq \prod_{i=1}^n\frac{K}{(1+k_i)^{b-\frac{\sigma(\j)}{2n}}},\vspace{10pt}\\
\text{ and } \vspace{5pt}\\
       \exists M>0,\, \forall {\k}\in \N^n, \vspace{5pt}\\ \ds \Bigg|\int_{\R^n}\varphi(\z)e^{-\theta\|\z+(c\tau+\ell)\u\|^2}\Gamma_{\k}(\z)\,\mathrm d\z\Bigg |\le \prod_{i=1}^n{M\over (1+k_i)^{b-\frac 1n}}
\end{array}\right.\right\},
\end{equation*}
where $D^\j=\partial_{z_1}^{j_1}\dots \partial_{z_n}^{j_n}$, and
\begin{equation*}
\mathcal Y:=\left\{ f\in C^{0}(\mathbb{R}^{n})\left|\begin{array}{c}
\exists C>0, \forall \z \in \R^n, \vert f (\z)\vert \leq C\ds\prod_{i=1}^n {1\over 1+z_i^2}, \vspace{10pt}\\
\text{ and } \vspace{5pt}\\
\ds \exists K>0, \forall \k \in \N^n, \left|\int_{\mathbb{R}^n}f(\z)\Gamma_{\k}(\z)d\z\right|\leq  \prod_{i=1}^n\frac{K}{(1+k_i)^{b-\frac 1 n}}
\end{array}\right.\right\}.
\end{equation*}
They are  Banach spaces, see \cite[Lemma 4.1]{AlfPel21} for a related situation, when equipped with their respective norm defined by
$$
\|(c,\varphi)\|_{\mathcal X}:= |c|+\sum_{\sigma(\j)\le 2}\, \sup_{\z\in \R^n} \left| D^{\j}\varphi (\z) \prod_{i=1}^n (1+z_i^2)^2\right|+\sum_{\sigma(\j)\le 2} \Vert D^{\j}\varphi\Vert_{b-\frac{\sigma(\j)}{2n}}+ \left\|\varphi e^{-\theta\|\bullet+(c\tau+\ell)\u\|^2}\right\|_{b-\frac 1n},
$$
and
$$
\|f\|_{\mathcal Y}:= \sup_{\z\in \R^n} \left|f(\z)\prod_{i=1}^n (1+z_i^2)\right|+\left\|f\right\|_{b-1/n},
$$
where, for $m\in \R$,
\[
\|w\|_{m}:=\sup_{\k\in \N^n}\,\left(\left|\int_{\R^n} w(\z)\Gamma_\k(\z)\,\mathrm d\z\right|\,\prod_{i=1}^n(1+k_i)^{m}\right).
\]

For  $\ep>0$, we are thus  looking for a perturbation of the triplet $(c^0=0,\varphi^0,H^0)$ defined in Lemma~\ref{lem:rhomaxzero}. Precisely we look for $c^\ep$, $\phi^\ep$, $\eta^\ep$ so that  the triplet
$$(c^\ep,\varphi^\ep,H^\ep):=(c^\varepsilon,\varphi^0 +\phi^\varepsilon,H^0 +\eta^\varepsilon)$$
validates the first equation in \eqref{eq:tw-2}. Recalling \eqref{3}, this means nothing else that 
 $\mathcal R(\ep,c^\varepsilon,\phi^\varepsilon,\eta^\ep)=0$ where
\begin{multline}
    \mathcal R(\varepsilon,c,\phi,\eta):=\mu_H^2 \Delta \phi +c\nabla \phi\cdot \u \nonumber \\
    +\left[R_H-\gamma_H(H^0+\eta)-\beta^2\|\z\|^2-\ep \left(R_P-{c^2\over 4\mu_P^2}-n\mu_P\alpha_P\right)\frac{H^0+\eta}{\gamma_P} e^{-\theta\|\z+(c \tau+\ell) \u\|^2}\right]\phi\nonumber \\
      +c\nabla \varphi^0(\z)\cdot \u -\gamma_H \eta \varphi ^0(\z) -\ep \left(R_P-{c^2\over 4\mu_P^2}-n\mu_P\alpha_P\right)\frac{H^0+\eta}{\gamma_P} e^{-\theta\|\z+(c\tau+\ell) \u\|^2}\varphi^0(\z).\label{def-mathcal-R}
\end{multline}
We thus consider the map
\begin{equation}\label{map-F}
\begin{array}{ccccc}
    \mathcal F & : & \R\times \mathcal X\times \R& \rightarrow &  \mathcal Y \times \R^2\\
     & &(\varepsilon,(c,\phi),\eta) & \mapsto & \ds
\left(\mathcal R (\varepsilon,c,\phi,\eta), \int_{\R^n} \phi(\z)\,\mathrm d\z, \int_{\R^n} z_1 \phi(\z)\,\mathrm d\z\right).
\end{array}
\end{equation}
Obviously $\mathcal F(0,0,0,0)=0$.  Let us now check the hypotheses of the Implicit Function Theorem.

\paragraph{$\mathcal F$ is well-defined.} From our choices of the  spaces $\mathcal X$ and $\mathcal Y$, for $(\varepsilon,c,\phi,\eta)\in \R \times \mathcal X\times \R$, the function $\mathcal R(\varepsilon,c,\phi,\eta)$ is well defined and in $\mathcal Y$, and $\mathcal{F}(\ep,c,\phi,\eta)$ is well defined and in $\mathcal Y \times \R^2$. This follows from a straightforward adaptation of \cite[Lemma 4.2]{AlfPel21}, the adequate linear material coming from Appendix \ref{A:linear}. Let us simply note that, because of the term $e^{-\theta \Vert \z+(c\tau+\ell) \u\Vert ^2}\phi$ appearing in the definition of $\mathcal R$ above, we had to choose $\mathcal X$ not as a single function space but as a \lq\lq scalar-function space'' (see the third condition in the definition of $\mathcal X$). This contrasts with the analysis in \cite{AlfPel21}.

\paragraph{$\mathcal L:=D_{(c,\phi,\eta)}\mathcal F (0,0,0,0)$ the Fr\'echet derivative of $\mathcal F$ along $(c,\phi,\eta)$ at point $(0,0,0,0)$.} It is straightforward to check that it exists and is given by the linear continuous operator
\begin{equation}
\label{def-mathcalL}
\mathcal L: (c,\phi,\eta)\in  \mathcal X\times \R \mapsto \left(R(c,\phi,\eta), \int_{\R^n} \phi(\z)\,\mathrm d\z,
    \int_{\R^n} z_1 \phi(\z)\,\mathrm d\z 
\right) \in \mathcal Y\times \R^2,
\end{equation}
where
\begin{eqnarray}
    R(c,\phi,\eta) &:=&\mu_H^2 \Delta \phi +\left(R_H-\gamma_HH^0-\beta^2\|\z\|^2\right)\phi+c \nabla \varphi^0(\z) \cdot \u -\gamma_H \eta\varphi^0(\z)\nonumber\\
    & =&\mu_H^2 \Delta \phi +\left(R_H-\gamma_HH^0-\beta^2\|\z\|^2\right)\phi-\,{c\beta\over \mu_H}\,z_1 \varphi^0(\z)  -\gamma_H \eta\varphi^0(\z),\label{def-R}
\end{eqnarray}
since $ \nabla \varphi^0(\z) = -\,\frac \beta {\mu_H} \,\varphi^0(\z)\,\z$ from Lemma \ref{lem:rhomaxzero} and $\u=(1,0,\dots,0)$. 

Also, one may check that $D_{(c,\phi,\eta)}\mathcal F$ is well-defined on a neighborhood of $(0,0,0,0)$, and that both $\mathcal F$ and $D_{(c,\phi,\eta)}\mathcal F$ are continuous at $(0,0,0,0)$.

The next two subsections are  dedicated to prove that the map $\mathcal L$ is one-to-one. This is the core of the proof. We proceed by analysis and synthesis.

\subsection{Construction of the antecedent}\label{ss:antecedent}

{Let $(f,h,r)\in \mathcal Y\times \R^2$ be given and assume there is   $(c,\phi,\eta)\in \mathcal X\times\R$ such that $\mathcal L (c,\phi,\eta)=(f,h,r)$.} By definition of $\mathcal X$ and $\mathcal Y$, the functions $\phi$ and $f$ belong to $L^2(\R^n)$ so we may use the Hilbert basis  $(\Gamma_\k)_{\k\in \N^n}$ of Proposition \ref{prop:basis_eigenfunctions} to write
\[
\phi(\z)= \sum_{\k\in \N^n} \phi_{\k} \Gamma_{\k}(\z),\qquad \hbox{ and } \qquad   f(\z)=\sum_{\k\in \N^n} f_{\k} \Gamma_{\k}(\z),
\]
where the coordinates $(f_\k)_{\k\in\N^n}$ are given, while 
 the coordinates $(\phi_\k)_{\k\in \N^n}$ are to be found. From Proposition \ref{prop:basis_eigenfunctions} and the expression of $H^0$ in Lemma \ref{lem:rhomaxzero}, we have
\begin{eqnarray*}
    \mu_H^2\Delta \phi-\beta^2\|\z\|^2\phi+\left(R_H-\gamma_HH^0\right)\phi &=&-\sum_{\k\in \N^n}\lambda_\k \phi_\k\Gamma_\k+n\mu_H\beta\phi\\
    & =& -2 \mu_H \beta\sum_{\k\in \N ^n}\sigma(\k)\phi_{\k}\Gamma_{\k}.
\end{eqnarray*}
Also, from the expression  of $\varphi^0$ in Lemma \ref{lem:rhomaxzero} and Proposition \ref{prop:basis_eigenfunctions}, we have
\begin{equation}
    \varphi^0(\z) = \left({\beta\over2\pi\mu_H}\right)^{n/2} {1\over C_{\0}}\Gamma_{\0}(\z) = \left({\beta\over 4\pi\mu_H}\right)^{n/4}\Gamma_{\0}(\z).
\label{eq:relation_varphiast_et_gamma0}
\end{equation}
From the above and \eqref{def-R}, we have
\begin{equation}
\label{important}
R(c,\phi,\eta)(\z)=-2 \mu_H \beta\sum_{\k\in \N ^n}\sigma(\k)\phi_{\k}\Gamma_{\k}(\z)-\,{c\beta\over \mu_H} \left({\beta\over 4\pi\mu_H}\right)^{n/4}\,z_1 \Gamma_{\0}(\z)  -\gamma_H \eta\left({\beta\over 4\pi\mu_H}\right)^{n/4}\Gamma_{\0}(\z).
\end{equation}
Hence, projecting $R(c,\phi,\eta)=f$ on all $\Gamma _k$'s, we collect the following relations according to $\k\in \N^n$.

\begin{enumerate}
    \item [(i)]  The case $\k=(0,\dots,0)$.   From Lemma \ref{lem:form_polyn_hermite} $(ii)$ we have $\int_{\R^n} z_1\Gamma_{\0}(\z)\Gamma_{\0}(\z)\, \mathrm d\z=0$  and thus $
        f_{\0} =-\gamma_H \eta \left({\beta\over 4\pi\mu_H}\right)^{n/4}$. In other words, $\eta$ is given by
    \begin{equation}
    \label{le-eta}
    \eta=-\, \frac{f_{\0}}{\gamma_H}\left(4\pi\mu_H\over \beta\right)^{n/4}.
    \end{equation}
    \item [(ii)] The case $\k=\u=(1,0,\dots,0)$. From Lemma \ref{lem:form_polyn_hermite} $(iii)$ we have $\int_{\R^n} z_1\Gamma_{\0}(\z)\Gamma_{\u}(\z)\,\mathrm d\z=\sqrt{\frac{\mu_H}{2\beta}}$ and thus
    $$
    f_{\u}=
    -2\mu_H\beta  \phi_\u   -\,{c\beta\over \mu_H}\left(\beta\over 4\pi\mu_H\right)^{n/4}\sqrt{\mu_H\over 2\beta}.
    $$
 In other words, $\phi_\u$ is given by (note that $c$ is still to be found)
    \begin{equation}
        \phi_{\u}=-\, {1\over 2\mu_H\beta}\, f_{\u} -\,{c\over  2\mu_H}\,{1\over \sqrt{2\mu_H\beta}} \left(\beta \over 4\pi \mu_H\right) ^{n/4}.
        \label{rel:varphi_u_et_c}
    \end{equation}
    \item [(iii)] The case $\k\in\N^n\setminus \{\0,\u\}$.  This time, from Lemma \ref{lem:form_polyn_hermite} $(iii)$ we have $\int_{\R^n} z_1\Gamma_{\0}(\z)\Gamma_{\k}(\z)\,\mathrm d\z=0$ and thus $\phi_\k$ is given by
    \begin{equation}
    \label{les-phi-k}
    \phi_{\k}=-{1\over 2 \mu_H\beta\sigma(\k)}\, f_{\k}, \quad \k\in\N^n\setminus \{\0,\u\}.
    \end{equation}
\end{enumerate}

Next, we need $h=\int_{\R^n}\phi(\z)\,\mathrm d\z$ which, in view of Lemma \ref{lem:form_polyn_hermite} $(i)$, translates to the equality $h=\sum _{\k \in \N ^n} \phi_{2\k}m_{2\k}$. In other words, $\phi_{\0}$ is given by
$$
\phi_{\0}=\frac{h-\sum_{\k\in \N^n\setminus \{\0\}}\phi_{2\mathbf k}m_{2\mathbf k}}{m_\0},
$$
or equivalently,  in view of Lemma \ref{lem:form_polyn_hermite} $(i)$ and \eqref{les-phi-k},
\begin{equation}\label{le-phi-zero}
\phi_{\0}=\left({\beta \over 4\mu_H\pi }\right)^{n/4}h+{1\over 4\mu_H\beta}\sum_{\k\in \N^n\setminus\{\0\}}\left( \prod_{i=1}^n {\sqrt{(2k_i)!}\over {2^{k_i}} k_i!}\right) \frac{f_{2\k}}{\sigma(\k)}.
\end{equation}
Note that, for $b>0$ large enough, the above series does converge. Indeed, from Stirling's formula there is $C>0$ such that ${\sqrt{(2k_i)!}\over {2^{k_i}} k_i!}\leq \frac{C}{k_i^{1/4}}\leq C$ and since, by the second condition in the definition of $\mathcal Y$, 
 $\vert f_\k\vert \leq \prod_{i=1}^n \frac{K}{(1+k_i)^{b-1/n}}$ one has  $\sum_{\k\in\N^n\setminus\{\0\}} \frac{\vert f_{2\k}\vert}{\sigma(\k)}<+\infty$ for $b>0$ sufficiently large.

Last, we need $r=\int_{\R^n} z_1\phi(\z)\,\mathrm d\z$  which, in view of Lemma \ref{lem:form_polyn_hermite} $(ii)$, translates to
\[r=\sum_{\begin{smallmatrix}
    \k=(k_1,\dots,k_n)\in \N^n\\
    k_1 \text{ odd},\\
     k_2,\dots,k_n\text{ even}
\end{smallmatrix}}\phi_\k w_{1,\k}.
\]In other words, in view of \eqref{rel:varphi_u_et_c}, $c$ is given by
$$
c=\left({1\over  2\mu_H}\,{1\over \sqrt{2\mu_H\beta}} \left(\beta \over 4\pi \mu_H\right) ^{n/4}\right)^{-1} \left(-\frac{1}{2\mu _H\beta}f_{\u}-\, {r\over w_{1,\u}}+\frac{1}{w_{1,\u}} \sum_{\begin{smallmatrix}
    \k=(k_1,\dots,k_n)\in \N^n \setminus \{\u\}\\
    k_1 \text{ odd},\\
     k_2,\dots,k_n\text{ even}
\end{smallmatrix}}\phi_\k w_{1,\k}\right).
$$
Using  Lemma \ref{lem:form_polyn_hermite} $(ii)$ and \eqref{les-phi-k}, some tedious but straightforward computations show that this can be recast
\begin{equation*}
    c =- 2\beta\mu_H r-2\sqrt{\mu_H\over \beta} \left(4\pi\mu_H\over \beta\right) ^{n/4}\sum_{\begin{smallmatrix}
    \k=(k_1,\dots,k_n)\in \N^n\\
    k_1 \text{ odd}\\
    k_2,\dots,k_n\text{ even}
\end{smallmatrix}} {\ds\prod_{i=1}^n \sqrt{k_i!}\over [(k_1-1)/2]!\ds\prod_{i=2}^n(k_i/2)!} {f_\k\over 2^{\sigma(\k)/2}\sigma(\k)},
\end{equation*}
or, equivalently,
\begin{equation}\label{le-c}
    c =- 2\beta\mu_H r-2\sqrt{\mu_H\over \beta} \left(4\pi\mu_H\over \beta\right) ^{n/4}\sum_{\begin{smallmatrix}
    \k=(k_1,\dots,k_n)\in \N^n\\
    k_1 \text{ odd}\\
    k_2,\dots,k_n\text{ even}
\end{smallmatrix}}\left( \frac{\sqrt{k_1!/2^{k_1}}}{[(k_1-1)/2]!}\prod_{i=2}^n \frac{\sqrt{k_i!/2^{k_i}}}{(k_i/2)!}\right)\frac{f_\k}{\sigma(\k)}.
\end{equation}
Under the above form, using very similar arguments as above, one can check that, for $b>0$ large enough, the above series does converge.

Hence, $(c,\phi,\eta)$ is uniquely determined by \eqref{le-eta}, \eqref{rel:varphi_u_et_c}, \eqref{les-phi-k}, \eqref{le-phi-zero}, and \eqref{le-c}.

\medskip

Conversely, we have to check that $(c,\phi)$ --- given by \eqref{rel:varphi_u_et_c}, \eqref{les-phi-k}, \eqref{le-phi-zero}, and \eqref{le-c}--- belongs to $\mathcal X$. In the sequel, $C$ denotes a positive constant depending on the different parameters ($\mu_H$, $\beta$, etc.) but not on the multi-index $\k\in \N^n$, and that may change from line to line. For all $\k\in \N^n\setminus\{\0, \u\}$, it follows from \eqref{les-phi-k}, the second condition in the definition of $\mathcal Y$ and Lemma~\ref{lem:eigenfunc_control}, that
$$
\Vert \phi _\k \Gamma _\k \Vert _\infty \leq C\frac{\vert f_\k\vert}{\sigma(\k)}\Vert \Gamma _\k\Vert_\infty \le {C\over \sigma(\k)}\prod_{i=1}^n {k_i^{1/4}\over (1+k_i)^{b-1/n}}\le  C \prod_{i=1}^n \frac 1{(1+k_i)^{b-1/4}}.
$$
Taking $b>0$ sufficiently large ensures the normal convergence of  $\sum_\k \phi_\k \Gamma_\k$ and, by very similar arguments, that of $\sum_\k \phi_\k \partial _{z_j}\Gamma_\k$, $\sum_\k  \phi_\k \partial _{z_j} \partial _{z_l}\Gamma_\k$. Hence $\phi\in C^2(\R^n)$. It remains to check the three conditions appearing in the definition of $\mathcal X$.

\paragraph{The function $\phi$ satisfies the first property in the definition of $\mathcal X$ (algebraically decay).} Using successively \eqref{les-phi-k}, estimates \eqref{norm_infini:gamma_k},  \eqref{norm_infini_a_poids} from  Lemma~\ref{lem:eigenfunc_control}, and the fact that $f$ satisfies the second condition in the definition of $\mathcal Y$, we get
\begin{align*}
    \left|\left(\phi(\z)-\varphi_\0\Gamma_\0(\z)-\phi_\u \Gamma_\u(\z)\right)\prod_{i=1}^n (1+z_i^2)^2\right| & \leq \sum_{\k\in \N^n\setminus\{\0,\u\}} \left|\phi_\k \left(\prod_{i=1}^n (1+z_i^2)^2\right)\Gamma_\k(\z)\right|\\
& \leq C \sum_{\k\in \N^n\setminus\{\0,\u\}}{\vert f_{\k}\vert \over \sigma(\k)}\prod_{i=1}^n k_i^{9/4}
\\
& \leq C \sum_{\k\in \N^n\setminus\{\0,\u\}}{1 \over \sigma(\k)}\prod_{i=1}^n {k_i^{9/4}\over (1+k_i)^{b-1/n}},
\end{align*}
which, again, is finite for  $b>0$ sufficiently large. Thus there is a constant $C>0$ such that
\begin{equation}
\label{truc}
\forall \z=(z_1,\dots,z_n)\in \R^n,\quad |\phi(\z)|\leq C \prod_{i=1}^n {1\over (1+z_i^2)^2}.
\end{equation}
Now let us fix $1\le i \le n$, with $k_i>0$ and define the multi-indexes
$$\k^-:=(k_1,\dots,k_{i-1},-1+k_i,k_{i+1},\dots,k_n), \quad \k^+:=(k_1,\dots,k_{i-1},1+k_i,k_{i+1},\dots,k_n).
$$
 By \eqref{def:eigenfunctionsGamma}, \eqref{eq:eigenfunc_cst} and the recursive relations \eqref{rec-der-H}, we have
\begin{align}
    \partial_{z_i}\Gamma_\k (z) & = C_\k \sqrt{\beta\over \mu_H} \,e^{-\, {\beta\over 2\mu_H}\|\z\|^2} \left[H'_{k_i}\left(\sqrt{\beta \over \mu_H}z_i\right)-\sqrt{\beta \over\mu_H}z_i H_{k_i}\left(\sqrt{\beta \over \mu_H}z_i\right)\right] \prod_{\begin{smallmatrix}
    j=1\\
    j\neq i
\end{smallmatrix}}^nH_{k_j}\left(\sqrt{\beta \over \mu_H} z_i\right) \nonumber \\
& =  C_\k \sqrt{\beta\over \mu_H} \,e^{-\, {\beta\over 2\mu_H}\|\z\|^2} \left[k_iH_{k_i-1}\left(\sqrt{\beta \over \mu_H}z_i\right)-{1\over2} H_{k_i+1}\left(\sqrt{\beta \over \mu_H}z_i\right)\right]\prod_{\begin{smallmatrix}
    j=1\\
    j\neq i
\end{smallmatrix}}^nH_{k_j}\left(\sqrt{\beta \over \mu_H} z_i\right) \nonumber \\
& =\sqrt{\beta\over \mu_H} \left(\sqrt{\frac{k_i}2}\Gamma_{\k^-}(\z)-\sqrt{\frac{k_i+1}2} \Gamma_{\k^+}(\z)\right).
\label{rel_deriv_Gamma_k}
\end{align}
Thanks to these relations, similar arguments as above show that, for any $\j\in \N^n$ with $\sigma(\j)\leq 2$, $D^{\j}\phi$ also satisfies \eqref{truc} (for $b>0$ sufficiently large).

\paragraph{The function $\phi$ satisfies the second property in the definition of $\mathcal X$ (decay of $L^2$ projections).}
For all $\k\in \N^n\setminus \{\0,\u\}$, it follows from \eqref{les-phi-k} and the fact that $f$ satisfies the second condition in the definition of $\mathcal Y$ that
\[
\left|\int_{\R^n} \phi(\z) \Gamma_\k(\z) \, \mathrm{d}\z \right|=|\phi_\k| \leq C { |f_\k|\over \sigma(\k)}\le C\frac{\prod_{i=1}^n (1+k_i)^{1/n}}{\sigma(\k)} \prod_{i=1}^n{1\over (1+k_i)^{b}} \leq  \prod_{i=1}^n{C\over (1+k_i)^{b}},
\]
since $\sigma(\k)=\prod_{i=1}^n \sigma(\k)^{1/n}\geq \prod _{i=1}^n k_i^{1/n}$. Next,  integration by parts and \eqref{rel_deriv_Gamma_k} yields, for all $1\le i\le n$,
\begin{align*}
    \left|\int_{\R^n}  \partial_{z_i}\phi(\z) \Gamma_\k(\z) \, \mathrm{d}\z \right|& =\left|\int_{\R^n}  \phi(\z) \partial_{z_i}\Gamma_\k(\z) \, \mathrm{d}\z \right|\\
     & \leq C\left( \sqrt{k_i}\vert \phi_{\k^-}\vert +\sqrt{k_i+1}\vert \phi_{\k^+}\vert \right)\\
    & \leq C\left( \sqrt{k_i}{\vert f_{\k^-}\vert \over \sigma(\k^-)}+\sqrt{k_i+1}{\vert f_{\k^+}\vert \over \sigma(\k^+)}\right)\\
      & \le C\prod_{\begin{smallmatrix}
        l=1
    \end{smallmatrix}}^n{1\over (1+k_l)^{b-\frac 1{2n}}},
\end{align*}
since  $\sigma(\k^\pm)=\sigma(\k^\pm)^{1/2}\sigma(\k^\pm)^{1/2}\ge C k_i^{1/2} \prod_{l=1}^n k_l^{1/(2n)}$ and $f\in \mathcal Y$. Very similar arguments reveal that
\[
\left|\int_{\R^n} D^{\j}\phi(\z) \Gamma_\k(\z)\, \mathrm d \z\right|\le \prod_{i=1}^n{C\over (1+k_i)^{b-\frac 1 n}},
\]
for all $\j\in \N^n$ with $\sigma(\j)=2$, which concludes this part.

\medskip

Last, the fact that $\phi$ satisfies the third property (involving the speed $c$) in the definition of $\mathcal X$  deserves a full subsection.

\subsection{The function \texorpdfstring{$\phi$}{} satisfies the third property in the definition of \texorpdfstring{$\mathcal X$}{}}

Let us fix $\k=(k_1,\dots,k_n)\in \N^n$.  For the sake of clarity, we use the shortcuts
\begin{equation}\label{shortcuts}
\ot := {\mu_H\over \beta}\, \theta >0, \quad \oc: = \sqrt{\beta\over \mu_H}\,c, \quad  \ol = \sqrt{\beta\over \mu_H}\,\ell,
\end{equation}
Using successively that $\phi(\z) =\sum_{\j\in \N^n} \phi _\j \Gamma _\j (\z)$, Fubini's theorem, \eqref{def:eigenfunctionsGamma} and \eqref{eq:eigenfunc_cst}, we get
\begin{align*}
    \left|\int_{\R^n}\phi(\z)\right.& \left.e^{-\theta\|\z+(c\tau+\ell)\u\|^2}\Gamma_{\k}(\z)\,\mathrm d\z \right|= \left|\sum_{\j \in \N^n}\phi_\j \int_{\R^n}\Gamma_\j (\z)\Gamma_\k (\z) e^{-\theta \|\z+(c\tau+\ell) \u \|^2}\mathrm d\z\right|\\
    & =\left|C_\k \, \sum_{\j \in \N^n} \phi_\j\,  C_\j  \, \prod_{i=1}^n\int_{\R}H_{j_i} \left(\sqrt{\beta\over \mu_H}z_i\right)H_{k_i}\left(\sqrt{\beta\over \mu_H}z_i\right) e^{-\, {\beta\over \mu_H}z_i^2-\theta (z_i+(c\tau+\ell) u_i)^2}\mathrm dz_i\right|\\
    & =\left|\pi^{-\frac n2} \, \sum_{\j \in \N^n} \,  {\phi_\j\over \sqrt{2^{\sigma(\k)+\sigma(\j)}\prod_{i=1}^nk_i!j_i!}}\prod_{i=1}^n\int_{\R}H_{j_i} (y)H_{k_i}(y) e^{-y^2-\theta {\mu_H\over \beta}\left(y+\sqrt{\beta\over \mu_H}(c\tau+\ell) u_i\right)^2}\mathrm dy\right|\\
    & \le \pi^{-\frac n2} \sum_{\j \in \N^n} {|\phi_\j|\over \sqrt{2^{\sigma(\k)+\sigma(\j)}\prod_{i=1}^nk_i!j_i!}} \prod_{i=1}^n\left|\int_{\R}H_{j_i} (y)H_{k_i}(y) e^{-y^2-\ot\left(y+(\oc\tau+\ol) u_i\right)^2}\mathrm dy\right|\\
    & =:\pi^{-\frac n 2}\sum _{\j\in \N^n} \gamma_\j ^{\k}.
\end{align*}
We now investigate on $\gamma_\j^{\k}$, $\j \in \N^n$. In the sequel, $C$ denotes a positive constant depending on the different parameters ($\mu_H$, $\beta$, etc.),   but not on the multi-indexes $\k$, $\j\in \N^n$, and that may change from line to line.

\paragraph{First case: $\j=\0$.} Thanks to Lemma~\ref{lem:prod_scal_Hk_Hj_with_exp}, as $H_0\equiv 1$, we have, denoting $\kappa=-(\oc\tau+\ol) u_i$,
\begin{align*}
    \left|\int_{\R}H_{0} (y)H_{k_i}(y) e^{-y^2-\ot\left(y+(\oc\tau+\ol) u_i\right)^2}\mathrm dy\right|& = \sqrt{\pi\over 1+\ot}\left(\ot\over 1+\ot\right)^{k_i/2}\exp\left(-\, {\ot\, \,\kappa ^2\over 1+\ot}\right)\left|H_{k_i}\left(\sqrt{\ot\over1+\ot}\kappa\right)\right|\\
    & \leq C \left(\ot\over 1+\ot\right)^{k_i/2} 2^{k_i/2}\sqrt{k_i!} \\
    & \leq C  2^{k_i/2}\sqrt{k_i!},
\end{align*}
where we have used Cramer's inequality for Hermite polynomials, namely
\begin{equation}
\forall k\in \N,\, \forall x\in \R,\qquad    |H_k(x)|\leq C \sqrt{2^kk!}\, e^{\frac{x^2}2}.
\label{cramers_ineq}
\end{equation}
As a result,
\[
\gamma_\0 ^{\k}={|\phi_\0|\over \sqrt{2^{\sigma(\k)}\prod_{i=1}^nk_i!}}\prod_{i=1}^n\left|\int_{\R}H_{0} (y)H_{k_i}(y) e^{-y^2-\ot\left(y+(\oc\tau+\ol) u_i\right)^2}\mathrm dy\right|\leq C\vert \phi_\0\vert \leq \prod_{i=1}^n{C\over (1+k_i)^b},
\]
since $\phi$ satisfies the second property in the definition of $\mathcal X$. It thus remains to  estimate   $\sum_{\j \neq \0} \gamma _\j^\k$.

\paragraph{Second case: $\j\neq \0$.} With a slight abuse of notation, we now redefine $f_\u$ as the quantity such that
\[
\phi_\u=-\, {f_\u\over 2\mu_H\beta\sigma(\u)},
\]
see \eqref{rel:varphi_u_et_c}. This enables to make the relation \eqref{les-phi-k} valid for any non trivial multi-index and, in particular, 
$$
\sum_{\j \in \N^n\setminus \{\0\}} \gamma _\j ^\k\leq C s_c(\k),
$$
where
$$
s_c(\k):=\sum_{\j\in \N^n\setminus \{\0\}} {|f_\j|\over \sigma(\j)}\, \prod_{i=1}^n {1\over\sqrt{2^{k_i+j_i} k_i!j_i!}}\left|\int_{\R}H_{j_i}(y)H_{k_i}(y)e^{-y^2 -\ot(y+(\oc\tau+\ol) u_i)^2}\,\mathrm d y\right|.
$$
Since $f\in \mathcal Y$, we have 
\[
s_c(\k)\leq C\sum_{\j\in \N^n\setminus\{\0\}}\prod_{i=1}^n{1\over \sqrt{2^{k_i+j_i}k_i!j_i!}} \, {1\over p(j_i)^{1/n}(1+j_i)^{b-1/n}}\, \left|\int_{\R}H_{j_i} (y)H_{k_i}(y) e^{-y^2-\ot\left(y+(\oc\tau+\ol) u_i\right)^2}\mathrm dy\right|,
\]
where
\[
p(j)=\left\{\begin{array}{ll}
    1, & \text{ if $j= 0$},  \\
    j, & \text{ if $j>0$}.
\end{array}\right.
\]
We have finally simplified the problem into the study of a one-dimensional indexed series, depending on an integer $k$ and a real $u\in \{0,1\}$, namely
\begin{equation}
\Sigma(k,u):=\sum_{j=0}^{+\infty}\underbrace{{1\over \sqrt{2^{k+j}k!j!}} \, {1\over p(j)^{1/n}(1+j)^{b-1/n}}\, \left|\int_{\R}H_{j} (y)H_{k}(y) e^{-y^2-\ot\left(y+(\oc\tau+\ol) u\right)^2}\mathrm dy\right|}_{:=\gamma_j^k},
\label{maj:sc_avec_sigma_i}
\end{equation}
as $s_c(\k)\le C\prod_{i=1}^n \Sigma (k_i,u_i)$. In the sequel, we fix $u\in \{0,1\}$ (recall \eqref{choix-u}), and focus on $\Sigma(k,u)$. 

Applying Lemma~\ref{lem:prod_scal_Hk_Hj_with_exp}, we get
\begin{multline*}
\gamma_j^k\leq {C \over \sqrt{2^{k+j}k!j!}} \left(1+j\over p(j)\right)^{1/n} \frac 1{(1+j)^{b}} \left(\ot\over 1+\ot\right)^{\frac{j+k}2} \exp\left(-\, {\ot \kappa^2 \over 1+\ot}\right)\\
    \times \sum_{l=0}^{\min(k,j)} {j! k!\over l!{(j-l)!}(k-l)!} \left({{2}\over \ot}\right)^l    \left|H_{k-l}\left(\sqrt{\ot\over 1+\ot}\kappa\right)\right|\left|H_{j-l}\left(\sqrt{\ot\over 1+\ot}\kappa\right)\right|,
\end{multline*}
so that, using the shortcut $\lambda=\sqrt{\ot\over 1+\ot}$,
$$
\gamma_j^k\leq {C \over \sqrt{2^{k+j}}}  \frac 1{(1+j)^{b}} \lambda^j \lambda^k e^{- \lambda^2\kappa^2}\sum_{l=0}^{\min(k,j)} {\sqrt{j! k!}\over l!{(j-l)!}(k-l)!} \left({{2}\over \ot}\right)^l    \left|H_{k-l}\left(\lambda\kappa\right)\right|\left|H_{j-l}\left(\lambda\kappa\right)\right|.
$$
At this point we need a (local) refinement of the (uniform) Cramer's inequality \eqref{cramers_ineq}, namely \cite[Theorem 1, $(i)$]{Bon-Cla-90} (see also \cite{Erd-60}) which, in our setting, rewrites as
\begin{equation}
\forall k\in \N, \qquad    |H_k(\lambda \kappa)|\leq C \frac{\sqrt{2^kk!}}{\max(1,k^{1/4})} \, e^{\frac{\lambda^2\kappa^2}2}.
\end{equation}
Using this twice we reach
\begin{equation}
\label{rename-gamma}
\gamma_j^k\leq \frac{C}{(1+j)^b}\lambda^j\lambda^k \sum_{l=0}^{\min(k,j)}\frac{\sqrt{\binom{j}{l}\binom{k}{l}}}{\max(1,(k-l)^{1/4})\times \max(1,(j-l)^{1/4})}(\ot^{-1})^l.
\end{equation}
Based on this, we prove in Appendix \ref{B:proof-taupe} the following. 

\begin{lem}\label{lem:taupe-plus} There holds
\begin{equation}
\label{limsup}
\limsup_{k\to+\infty}\, (1+k)^{b-\frac 12} \sum_{j=0}^{+\infty} \gamma_j^k<+\infty.
\end{equation}
\end{lem}

We have finally proven that
\[
\Sigma(k,u)\le {C\over (1+k)^{b-\frac 12}},
\]
which ends the proof that $\phi\in \mathcal X$,  at least when $n=1$ or $n=2$. 

\begin{rem} \label{rem:num-conj}  A numerical conjecture is that a stronger property than \eqref{limsup} holds true, namely that, for any $n\geq 1$,
\begin{equation}
\label{limsup-conj}
\limsup_{k\to+\infty}\, (1+k)^{b-\frac 1n} \sum_{j=0}^{+\infty} \gamma_j^k<+\infty.
\end{equation}
If so then our whole construction works in any dimension $n\geq 1$. We leave this as an open question.
\end{rem}

\subsection{End of the proof}

 Hence, from the above we can apply the Function Implicit Theorem: there exists $\ep^*>0$ such that for all $0\leq \ep<\ep^*$, there exists $(c^\ep,\phi^\ep,\eta^\ep)\in \mathcal X\times  \R$ such that $(c^\ep,H^\ep,\varphi^\ep)=(c^\ep,H^0+\eta^\ep,\varphi^0+\phi^\ep)$ satisfies the first equation in \eqref{eq:tw-2}, namely
\begin{equation}\label{eq-satisfied}
-c^\ep \nabla \varphi^\ep \cdot \u  = \mu_H^2 \Delta \varphi^\ep + \left(R_H-\gamma_H H^\ep -\beta^2 \Vert \z \Vert ^2-\ep P^\ep e^{-\theta\Vert \z+(c^\ep\tau+\ell) \u \Vert ^2 }\right) \varphi ^\ep,   \qquad \z\in \R^n,
\end{equation}
where
\[
P^\ep:=\frac 1{\gamma_P} \left(R_P-n\mu_P\alpha_P-\frac{(c^\ep)^2}{4\mu_P^2}\right)H^\ep.
\]
Also, up to reducing $\ep^*$, we have $H^\ep>0$ (since $H^0>0$),  $P^\ep>0$ (since $R_P-n\mu_P\alpha_P>0$).

To conclude the proof of Theorem \ref{th:lin_pursuit}, it remains to prove the positivity of the  speed $c^\ep$, of the  profile $\varphi^\ep$, and the negativity of the mass $\eta^\ep$ (meaning $H^\ep<H^0$ for $\ep>0$). 

\paragraph{The sign of $\eta^\ep$ and of $c^\ep$  for $0<\ep\ll 1$.}  Let us recall that the map $\mathcal F$ was defined in \eqref{map-F}. Differentiating the equality $\mathcal{F}(\ep,c^\ep,\phi^\ep,\eta^\ep)=0$ with respect to $\ep$ and evaluating at $\ep=0$, we get
$$
\frac{\partial \mathcal F}{\partial \ep}(0,0,0,0)+\mathcal L\left(\left.\frac{d c^\ep}{d\ep} \right\vert_{\ep=0}, \left.\frac{d \phi^\ep}{d\ep}\right\vert_{\ep=0}, \left.\frac{d \eta^\ep}{d\ep} \right\vert_{\ep=0} \right)=0_{\mathcal Y\times \R^2},
$$
where $\mathcal L$ was defined in \eqref{def-mathcalL}. Computing 
$\frac{\partial \mathcal F}{\partial \ep}(0,0,0,0)$, using $H^0={R_H-n\mu_H\beta\over \gamma_H}$ and \eqref{eq:relation_varphiast_et_gamma0}, this transfers to
\begin{equation}\label{to-be-inverted}
\mathcal L\left(\left.\frac{d c^\ep}{d\ep} \right\vert_{\ep=0}, \left.\frac{d \phi^\ep}{d\ep}\right\vert_{\ep=0}, \left.\frac{d \eta^\ep}{d\ep} \right\vert_{\ep=0} \right)=(f,0,0),
\end{equation}
with
$$
f(\z):=M\left(\frac{\beta}{4\pi \mu_H}\right)^{\frac n 4}e^{-\theta\|\z +\ell \u\|^2} \Gamma _\0(\z), \quad M:={(R_H-n\mu_H\beta)(R_P-n\mu_P\alpha_P) \over \gamma_H \gamma_P}>0.
$$
At this point we take advantage of subsection \ref{ss:antecedent} to invert \eqref{to-be-inverted}.

First, \eqref{le-eta} provides 
           \begin{eqnarray*}
 \left.\frac{d \eta^\ep}{d\ep} \right\vert_{\ep=0}&=&\, -\frac{f_{\0}}{\gamma_H}\left(4\pi\mu_H\over \beta\right)^{n/4}\\
 &=&
      - {M \over \gamma_H}\int_{\R^n} e^{-\theta\|\z+\ell \u\|^2} \Gamma _\0 ^2(\z) \,\mathrm d \z\\
       &=&
    -   {M \over \gamma_H}
    \left(\frac{\beta}{\pi \mu _H}\right)^{\frac n 2}   
       \int_{\R^n} e^{-\theta\|\z+\ell \u\|^2} e^{-\frac{\beta}{\mu_H}\|\z\|^2}\,\mathrm d \z\\
       &=&
-       {M \over \gamma_H}\left({\beta \over \beta +\mu_H\theta}\right)^{\frac n 2}e^{-\frac{\beta\theta\ell^2}{\beta+\mu_H\theta}},
     \end{eqnarray*}
     using straightforward computations based on Lemma \ref{lem:gaussiennes}. In particular  $
\left.\frac{d \eta^\ep}{d\ep} \right\vert_{\ep=0}<0$, insuring that, up to reducing $\ep^*>0$ if necessary,  $\eta^\ep<0$  for any $0<\ep<\ep^*$.

   Next, \eqref{le-c} provides  
    \begin{equation}\label{le-c-bis}
    \left.\frac{d c^\ep}{d\ep} \right\vert_{\ep=0} =-2\sqrt{\mu_H\over \beta} \left(4\pi\mu_H\over \beta\right) ^{n/4}\sum_{\begin{smallmatrix}
    \k=(k_1,\dots,k_n)\in \N^n\\
    k_1 \text{ odd}\\
    k_2,\dots,k_n\text{ even}
\end{smallmatrix}}\left( \frac{\sqrt{k_1!/2^{k_1}}}{[(k_1-1)/2]!}\prod_{i=2}^n \frac{\sqrt{k_i!/2^{k_i}}}{(k_i/2)!}\right)\frac{f_\k}{\sigma(\k)}.
\end{equation}
Our task is now to show that this quantity is positive (so that, up to reducing $\ep^*>0$ if necessary, $c^\ep>0$ for any $0<\ep<\ep^*$). To do so, for 
\begin{equation}\label{k-j}
\k=(2j_1+1,2j_2,\dots,2j_n),
\end{equation}
we focus on $f_\k$. We shall use the notation $A_\k\propto B_\k$ when $A_\k=C B_\k$ for some constant $C>0$ depending on $M$, parameters $n$, $\beta$, $\mu_H$, $\theta$, but not on $\ell$ (that we need to take small enough at some point) nor on the multi-index $\k$. Hence, recalling Proposition \ref{prop:basis_eigenfunctions},
\begin{eqnarray*}
f_\k&\propto& \int_{\R^n} e^{-\theta\|\z +\ell \u\|^2} \Gamma _\0(\z)\Gamma_\k(\z)\,\mathrm d \z\\
&\propto & C_\k \int_\R e^{-\theta(z_1+\ell)^2}e^{-\frac{\beta}{\mu_H}z_1^2}H_{2j_1+1}\left(\sqrt{\frac \beta \mu_H}z_1\right)\,\mathrm d z_1\prod _{i=2}^n\int _\R e^{-\theta z_i^2}e^{-\frac{\beta}{\mu_H}z_1^2}H_{2j_i}\left(\sqrt{\frac \beta \mu_H}z_i\right)\,\mathrm d z_i \\
&\propto & C_\k  \int_\R e^{-\ot(y+\ol)^2}e^{-y^2}H_{2j_1+1}\left(y\right)\,\mathrm d y \prod _{i=2}^n\int _\R e^{-\ot y^2}e^{-y^2}H_{2j_i}\left(y\right)\,\mathrm d y,
\end{eqnarray*}
 using again the shortcuts \eqref{shortcuts}. Note that the assumption $\ell>0$ is here crucial since $\ell=0$ would imply the vanishing of the first integral in the above term.  In view of Lemma \ref{lem:prod_scal_Hk_Hj_with_exp} (with $k=0$, $j=2j_1+1$, $\kappa=-\ol$ and with  $k=0$, $j=2j_i$, $\kappa=0$), this yields
\begin{eqnarray*}
f_\k&\propto & e^{-\lambda^2\ol^2} C_\k \lambda^{2j_1+1}H_{2j_1+1}(-\lambda\ol)\prod_{i=2}^n (-\lambda^2)^{j_i},
\end{eqnarray*}
with the shortcut $\lambda=\sqrt{\frac{\ot}{1+\ot}}$. Plugging this into \eqref{le-c-bis} and using \eqref{eq:eigenfunc_cst}, we reach (note that $\sigma(\k)=2\sigma(\j)+1$) 
 \begin{eqnarray*}   
  - e^{\lambda^2\ol^2} \left.\frac{d c^\ep}{d\ep} \right\vert_{\ep=0}&\propto &\sum_{j_2,\dots,j_n} \prod_{i=2}^n \frac{1}{j_i!}\left(-\frac{\lambda^2}4\right)^{j_i}\sum _{j_1=0}^{+\infty}\frac{1}{2\sigma(\j)+1}\frac{1}{j_1!}\left(\frac{\lambda}{2}\right)^{2j_1+1}H_{2j_1+1}(-\lambda\ol).
\end{eqnarray*}
Using the expression of the odd Hermite polynomials, namely
$$
H_{2j+1}(x)=(2j+1)!\sum_{m=0}^j \frac{(-1)^{j-m}}{(2m+1)!(j-m)!}(2x)^{2m+1},
$$
and Fubini theorem, we get
 \begin{eqnarray*}   
  - e^{\lambda^2\ol^2} \left.\frac{d c^\ep}{d\ep} \right\vert_{\ep=0}&\propto & \sum_{m=0}^{+\infty} (-1)^m\frac{(-2\lambda\ol)^{2m+1}}{(2m+1)!}  \alpha_m,
\end{eqnarray*}
where
$$
\alpha_m:= \sum_{j_2,\dots,j_n} \prod_{i=2}^n \frac{1}{j_i!}\left(-\frac{\lambda^2}4\right)^{j_i}\sum _{j_1=m}^{+\infty}\frac{1}{2\sigma(\j)+1}\frac{1}{j_1!}\left(\frac{\lambda}{2}\right)^{2j_1+1}\frac{(2j_1+1)!(-1)^{j_1}}{(j_1-m)!}.
  $$
As announced above we now turn to sufficiently small $\ell\propto \ol$: the above provides the equivalent (still up to a positive multiplicative constant)
 \begin{eqnarray*}   
  \left.\frac{d c^\ep}{d\ep} \right\vert_{\ep=0}&\sim & 2\lambda\ell
  \alpha_0, \quad \text{ as } \ell \to 0.
\end{eqnarray*}
Hence $\left.\frac{d c^\ep}{d\ep} \right\vert_{\ep=0}>0$ for sufficiently small $\ell$ (see the setting of Theorem \ref{th:lin_pursuit}) provided $\alpha_0>0$ which we now aim at proving. If $n=1$, this reduces to
$$
\alpha_0=\sum _{j=0}^{+\infty}(-1)^j\frac{(2j)!}{(j!)^2}\left(\frac{\lambda}{ 2}\right)^{2j+1}
$$
an alternating series which, as easily checked, is positive since $\lambda^2<1$. Hence, it only remains to consider the case $n\geq 2$. To so do, we rewrite $\alpha_0$ as 
$$
\alpha_0=\sum_{\j=(j_1,\dots,j_n)} (-1)^{\sigma(\j)}g(\j), \qquad g(\j):=\frac{1}{2\sigma(\j)+1}\frac{(2j_1+1)!}{(j_1!)^2}\left(\frac{\lambda}{2}\right)^{2j_1+1}\prod_{i=2}^n \frac{1}{j_i!}\left(\frac{\lambda^2}{4}\right)^{j_i}>0
$$
to side with the work \cite{Bor-Bor-86} on alternating series in several dimensions. 
We claim that, for any $\j\in \N^n$, 
\begin{equation}\label{claim}
   g(\j)> \sum_{\begin{smallmatrix}\k\in\{0,1\}^n\\\text{the number of $1$ is odd} \end{smallmatrix}} g(\j+\k)=:S(\j).
\end{equation}
This fact, whose proof is postponed, means that the sign of the alternating sum over any unit $n$-cube is decided by the term at the corner nearest the origin, and immediately implies that $g$ is $n$-monotone decreasing in the sense of \cite[Section 2]{Bor-Bor-86}, and, similarly, so are all its coordinate restrictions. In other words, $g$ is fully monotone in the sense of \cite[Section 2]{Bor-Bor-86} and the positivity of the alternating sum $\alpha_0$ is provided by \cite[Theorem 3.1]{Bor-Bor-86}. To conclude, we prove \eqref{claim}. We have 
\begin{eqnarray*}
    S(\j) &\leq & \frac{1}{2\sigma(\j)+3}\Bigg(\sum_{\begin{smallmatrix}\k\in\{0,1\}^n \\\text{the number of $1$ is odd}\\k_1=0\end{smallmatrix}}\frac{(2j_1+1)!}{(j_1!)^2}\left(\frac{\lambda}{2}\right)^{2j_1+1}\prod_{i=2}^n \frac{1}{(j_i+k_i)!}\left(\frac{\lambda^2}{4}\right)^{j_i+k_i}\\
    &&\qquad \qquad\qquad  +\sum_{\begin{smallmatrix}\k\in\{0,1\}^n\\\text{the number of $1$ is odd}\\k_1=1\end{smallmatrix}}\frac{(2j_1+3)!}{((j_1+1)!)^2}\left(\frac{\lambda}{2}\right)^{2j_1+3}\prod_{i=2}^n \frac{1}{(j_i+k_i)!}\left(\frac{\lambda^2}{4}\right)^{j_i+k_i}\Bigg),
\end{eqnarray*}
and thus (note that  $3\times 2^{n-2}\lambda^2<1$ from the assumption $\beta >(3\times2^{n-2}-1)\mu_H\theta$ in Theorem \ref{th:lin_pursuit})
\begin{eqnarray*}
    S(\j) &\leq &
 \frac{1}{2\sigma(\j)+3}\Bigg(\frac{\lambda^2}{4}\sum_{\begin{smallmatrix}\k\in\{0,1\}^n \\\text{the number of $1$ is odd}\\k_1=0\end{smallmatrix}}\frac{(2j_1+1)!}{(j_1!)^2}\left(\frac{\lambda}{2}\right)^{2j_1+1}\prod_{i=2}^n \frac{1}{(j_i)!}\left(\frac{\lambda^2}{4}\right)^{j_i}\\
    &&\qquad \qquad \qquad +\frac 3 2 \lambda^2 \sum_{\begin{smallmatrix}\k\in\{0,1\}^n\\\text{the number of $1$ is odd}\\k_1=1\end{smallmatrix}}\frac{(2j_1+1)!}{(j_1!)^2}\left(\frac{\lambda}{2}\right)^{2j_1+1}\prod_{i=2}^n \frac{1}{(j_i)!}\left(\frac{\lambda^2}{4}\right)^{j_i}\Bigg)\\
    &\leq & \frac{1}{2\sigma(\j)+3} \times \frac 3 2 \lambda^2 \times 2^{n-1}
g(\j) (2\sigma(\j)+1)\\
&<& g(\j),
\end{eqnarray*}
and we are done.

\paragraph{The profile $\varphi^\ep$ is positive for $0<\ep\ll 1$.}  Assume by contradiction that there is a sequence $\ep_p\searrow 0$ such that $\varphi^{\varepsilon_{p}}$ is {\it not} nonnegative  on $\mathbb{R}^{n}$. Since $\lim_{\Vert \z\Vert\to +\infty} \varphi^{\ep_p}(\z)=0$, there is a point $\z^p$ where $\varphi^{\ep_p}$ reaches its negative minimum.  Testing \eqref{eq-satisfied} at point $\z^p$ we get
$$
\beta^2\Vert \z^p\Vert ^2\leq R_H-\gamma_H H^{\ep_p}-\ep_p P^{\ep _p}e^{-\theta \Vert \z^p+(c^{\ep_p}\tau+\ell) \u\Vert ^2}\leq R_H.
$$
However, let us underline that $\varphi^{0}>0$ on $\mathbb{R}^n$
and  $\Vert \varphi^{\varepsilon}-\varphi^{0}\Vert_{L^{\infty}(\mathbb{R}^{n})}\rightarrow0$ as $\ep \rightarrow0$. As a consequence, for $p$ large enough, there
holds $\varphi^{\varepsilon_p}(\z)>0$ for all $\z$ such that $\Vert \z\Vert \leq \frac{\sqrt{R_H}}{\beta}$, which contradicts $\varphi^{\ep_p}(\z^p)<0$. Therefore, by reducing $\ep ^*>0$ if necessary, we have that, for all $0\leq\ep<\ep^*$, $\varphi^{\ep}$ is nonnegative and thus, from the strong maximum principle, positive.

\medskip

This concludes the proof of Theorem \ref{th:lin_pursuit}. \qed

\appendix

\section{The multivariate harmonic oscillator}\label{A:linear}

We fix $\beta>0$ and consider the multivariate harmonic oscillator operator acting on $u:\R^n\to \R$ through
\[
\mathcal Hu := -\mu_H^2\Delta u+\beta^2\|\z\|^2u.
\]
We define the $L^1$-norm $\sigma(\k)$ of a multi-index $\k$ as
\begin{equation}
    \sigma(\k):=\sum_{i=1}^n k_i,  \quad \text{ for }  \k=(k_1,\dots,k_n) \in \N^n.
\end{equation}
We denote $(H_{i})_{i\in\mathbb{N}}$  the family of Hermite polynomials, that is the unique family of real polynomials satisfying
\begin{equation}
\int_{\mathbb{R}}H_{i}(x)H_{j}(x)e^{-x^{2}}\,\mathrm dx=2^{i}i!\sqrt{\pi}\delta_{i,j},\qquad\deg H_{i}=i,\label{eq:eigenfunc_hermite}
\end{equation}
where $\delta$ stands for Kronecker delta.

The following is well-known, see \cite{PueTor18} or \cite{AlfPel21} among many others.

\begin{prop}[Eigenelements of the harmonic oscillator]
\label{prop:basis_eigenfunctions} The operator $\mathcal H:=-\mu_H^2\Delta +\beta^2\|\z\|^2$ admits a family of eigenelements $\left(\lambda_{\k},\Gamma_{\k}\right)_{\k\in \N^n}$, where, for any $\k=(k_1,\dots,k_n)\in \N^n$,
\begin{equation}
\lambda_{\k} =  \left(2\sigma(\k)+n\right) \mu_H \beta,\label{eq:eigenval}
\end{equation}
and
\begin{equation}\label{def:eigenfunctionsGamma}
\Gamma_{\k}(\z) = C_{\k} \exp\left({-\,{\beta\over 2\mu_H}}\|\z\|^2\right) \prod_{i=1} ^n H_{k_i}\left(\sqrt{\beta\over \mu_H}z_i\right),
\end{equation}
with
\begin{equation}
C_{\k}= \left(\beta\over \pi \mu_H\right)^{n/4}\, {\frac 1 {  \sqrt{ 2^{\sigma (\k)}\prod_{i=1}^n k_i!}}},
\label{eq:eigenfunc_cst}
\end{equation}
a normalization constant so that $\left\|\Gamma_{\k}\right\|_{L^2(\R^n)}=1$. Additionally, the family $(\Gamma_{\k})_{\k \in\mathbb{N}^n}$ forms a Hilbert
basis of $L^{2}(\mathbb{R}^n)$.
\end{prop}

The following elementary results will also be useful.

\begin{lem} \label{lem:form_polyn_hermite} Let $\k=(k_1,\dots,k_n)\in \N^n$ be a given multi-index.
\begin{enumerate}
    \item [(i)] The integral of $\Gamma_{\k}$ is given by
    \[
m_\k:=\int_{\R^n} \Gamma_{\k}(\z)\,\mathrm d\z =\left\{\begin{array}{ll}
    \ds \left(\pi\mu_H\over \beta\right)^{n/4}2^{(n-\sigma(\k))/2} \prod_{i=1}^n{\sqrt{k_i!}\over (k_i/2)!}, & \text{ if all $k_i$ are even,}\\
    0, & \text{ if not.} 
\end{array}\right.
\]
    \item [(ii)] Let $1\leq i \leq n$ be given. We define $\mathcal I_i=\{(k_1,\dots,k_n)\in \N^n$: $k_i$ is odd and for all $ j\neq i,\, k_j$ is even$\}$. Then the $i$-th mean value of $\Gamma_{\k}$ is given by
    \[
  w_{i,\k}:=  \int_{\R^n} z_i \Gamma_{\k}(\z) \,\mathrm d\z = \left\{\begin{array}{ll}
        \ds 2\sqrt{\mu_H\over \beta}\left(\pi\mu_H\over \beta\right)^{n/4}  {\ds 2^{(n-\sigma(\k))/2} \prod_{j=1}^n\sqrt{k_j!}\over \ds \,  [(k_i-1)/2]! \prod_{\begin{smallmatrix}j=1\\j\neq i\end{smallmatrix}}^n (k_j/2)! }, & \text{ if $\k\in \mathcal I_i$}, \\
        0, & \text{ if $\k \notin \mathcal I_i$.}
    \end{array}\right.
    \]
    \item [(iii)] For all $1\leq i\le n$, we have that
\[
\int_{\R^n} z_i\Gamma_{\k}(\z)\exp\left(-\,{\beta\over 2\mu_H}\, \|\z\|^2\right)\,\mathrm d\z= \sqrt{\mu_H\over 2\beta}\,\left(\mu_H \pi\over \beta\right)^{n/4}\, \delta_{k_i,1}\prod_{\begin{smallmatrix}j=1\\j\neq i\end{smallmatrix}}^n\delta_{k_j,0}.
\]
    \item [(iv)] For all $\theta >0$, we have that
\[
\int_{\R^n} \Gamma_{2\k}(\z)\exp\left(- \theta \|\z\|^2\right)\,\mathrm d\z= C_{2\k} \left(\frac{2\pi \mu_H}{\beta+2\theta \mu _H}\right)^{n/2}\,\prod_{i=1}^n\left(\beta-2\theta\mu_H\over \beta + 2\theta\mu_H \right)^{k_i}{(2k_i)!\over k_i!}.
\]
\end{enumerate}
\end{lem}

\begin{proof} From the well-known recursive relations on the Hermite polynomials
\begin{equation}\label{rec-der-H}
 H_{j+2}(x)=2x H_{j+1}(x)-2(j+1)H_{j}(x)\quad \hbox{ and } \quad H'_{j+1}(x)=2(j+1)H_j(x),
 \end{equation}
 one can check that
 \[
   I_j:= \int_{\R}H_j(x)e^{-\frac{x^2}{2}} \,\mathrm dx = \left\{\begin{array}{ll}
        \ds\sqrt{2\pi}\, {j!\over (j/2)!}, & \text{ if $j$ is even,} \\
        0, & \text{ if $j$ is odd,} 
    \end{array}\right.
    \]
    from which item $(i)$  follows from elementary computations, and that
   \[
 \int_{\R}x H_j(x)e^{-\frac{x^2}{2}} \,\mathrm dx = 2jI_{j-1},
    \]  
    from which item $(ii)$  follows from elementary computations.

Since $z_i\exp\left(-\,{\beta\over 2\mu_H} \|\z\|^2 \right)$ is nothing else than $\frac{1}{2 \sqrt{\frac{\beta}{\mu_H}}C_{\l}} \Gamma_{\l} (\z)$ where $\l=(0,\dots,0,1,0,\dots,0)$ with the 1 at the $i$-th position (because $H_1(x)=2x$), item $(iii)$ follows from the fact that $(\Gamma_{\k})_{\k \in\mathbb{N}^n}$ forms a Hilbert
basis of $L^{2}(\mathbb{R}^n)$.

Last, let us turn to the proof of item $(iv)$. From \eqref{def:eigenfunctionsGamma} and Fubini theorem, we easily see that the integral to be computed is equal to $C_{2\k} \left(\frac{\mu_H}{\beta}\right)^{n/2}\Pi _{i=1}^n J_{k_i}$ where
$$
J_p:=\int_\R H_{2p}(x)e^{-(\frac 12+\frac{\theta \mu_H}{\beta})x^2}\,dx, \quad p\in \N.
$$
Using the recursive relation in \eqref{rec-der-H}, integration by part and the second relation in \eqref{rec-der-H}, we straightforwardly reach
$$
J_p=2(2p-1)\frac{\beta-2\theta \mu_H}{\beta+2\theta \mu _H}J_{p-1},
$$
from which we deduce
$$
J_p=\left({\beta-2\theta\mu_H\over \beta +2\theta \mu_H}\right)^p\, {(2p)!\over p!}J_{0}=\left({\beta-2\theta\mu_H\over \beta +2\theta \mu_H}\right)^p\, {(2p)!\over p!} \sqrt{2\pi\beta \over \beta + 2\theta \mu_H },
$$
and the desired result.
\end{proof}

We also need the following $L^\infty$ and $L^1$ estimates which are less classical (since, obviously, the usual framework is $L^2(\R^n)$). 

\begin{lem}\label{lem:eigenfunc_control}  There is a constant $C=C(\mu_H,\beta)>0$ such that, for all $\k\in \N^n$,
\begin{align} 
 \|\Gamma_\k\|_{1} &\le C \prod_{i=1}^n k_i^{1/4},\label{norm_un:gamma_k}\\
   \|\Gamma_\k\|_{\infty} &\le C \prod_{i=1}^n k_i^{1/4}, \label{norm_infini:gamma_k}
\end{align}
and, for all indexes $1\le j,\, l\le n$,
\begin{equation} 
    \|\partial_{z_j}\Gamma_\k\|_{\infty}\le C k_j^{1/2}\prod_{i=1}^n k_i^{1/4}, \quad 
    \|\partial_{z_j}\partial_{z_l}\Gamma_\k\|_{\infty}\le C k_j^{1/2}k_l^{1/2}\prod_{i=1}^n k_i^{1/4},
\end{equation}
together with
\begin{equation}\label{norm_infini_a_poids}
||z_j^{2}\Gamma_{\k}||_{L^\infty}  \leq C k_j \prod_{i=1}^n k_i^{1/4},\quad  ||z_j^{4}\Gamma_{\k}||_{L^\infty}  \leq C k_j^2 \prod_{i=1}^n k_i^{1/4}.
\end{equation}
\end{lem}

\begin{proof} From \eqref{def:eigenfunctionsGamma}, we have $\Gamma _\k(\z)=\Pi_{i=1}^n \gamma_{k_i}(z_i)$ where $\gamma_{k_i}$ are the one dimensional eigenfunctions arising in  \cite[subsection 3.3]{AlfPel21} (with $A=\frac{\beta}{\mu_H}$). The conclusion is therefore a direct application of  \cite[Lemma 3.2]{AlfPel21}, which itself relies on \cite{AlfVer18}.
\end{proof}

Next, for $\gamma \in \R$, the Hermite polynomials are known to satisfy 
$$
H_k(\gamma x)=\sum_{i=0}^{\lfloor k/2\rfloor} \gamma^{k-2i} (\gamma^2 -1)^i {k!\over i!(k-2i)!}H_{k-2i}(x), \quad H_k(x+y)=\sum_{i=0}^k\left(\begin{matrix}
        k\\i
    \end{matrix}\right)(2y)^{k-i}H_i(x).
$$
As a result, for $0<\gamma<1$,
    \begin{align*}
        H_k(\gamma(x+y)) & =\sum_{i=0}^{\lfloor \frac k 2\rfloor}\gamma ^{k-2i}(\gamma^2-1)^i \, {k!\over i! (k-2i)!}\,\sum_{j=0}^{k-2i} {(k-2i)!\over j! (k-2i-j)!}(2y)^{k-2i-j}H_j(x)\\
        & = \sum_{j=0}^k \left[ \sum_{i=0}^{\lfloor {k-j\over 2}\rfloor}\gamma ^{k-2i}(\gamma^2-1)^i \, {k!\over i! j!(k-2i-j)!}(2y)^{k-2i-j}\right]H_j(x)\\
        & = \sum_{j=0}^k \left[{k!\over j!}\gamma ^k (2y)^{k-j}\sum_{i=0}^{\lfloor {k-j\over 2}\rfloor} {1\over  i!(k-j-2i)!}\left(\gamma^2-1\over 4y^2\gamma^2\right)^i \right]H_j(x)\\
        & = \sum_{j=0}^k \left[{k!\over j!(k-j)!}\, \gamma ^j  \left({1-\gamma^2}\right)^{(k-j)/2} H_{k-j} \left(y\gamma\over \sqrt{1-\gamma^2}\right) \right]H_j(x),
    \end{align*}
  since the Hermite polynomials are given by $
    H_k(x) = k!\sum_{i=0}^{\lfloor k/2\rfloor} {(-1)^i\over i!(k-2i)!}(2x)^{k-2i}$. We retain that, for all $k\in \N$, $0<\gamma <1$, $(x,y)\in \R^2$,
    \begin{equation}\label{mult_hermite_pol}   
 H_k(\gamma (x+y)) = \sum_{j=0}^k\left(\begin{matrix}
            k\\j
        \end{matrix}\right)\, \gamma ^j  \left({1-\gamma^2}\right)^{(k-j)/2} H_{k-j} \left(y\gamma\over \sqrt{1-\gamma^2}\right)H_j(x),
    \end{equation}
    which is useful to get the last following results. 

\begin{lem}\label{lem:prod_scal_Hk_Hj_with_exp}
    For all $j,\, k\in \N$ with $j\ge k$, and $\theta> 0$, $\kappa\in \R$, we have
\begin{multline*}
    \int_\R H_j(y)H_k(y)\exp\left(-y^2-\theta (y-\kappa)^2\right)\, \mathrm d y
        =\sqrt{\pi \over 1+\theta}\left(\theta\over 1+\theta\right)^{(j+k)/2}\exp\left(-\, {\theta\kappa^2\over 1+\theta}\right)
        \\ \times \sum_{l=0}^{k}{j!k!\over l! (j-l)!(k-l)!}\left(2\over \theta\right)^lH_{j-l}\left(\sqrt{\theta\over 1+\theta}\kappa\right)H_{k-l}\left(\sqrt{\theta\over 1+\theta}\kappa\right).
\end{multline*}
In particular, for $\kappa=0$, we get
\begin{multline*}
    \int_\R H_j(y)H_k(y)\exp\left(-(1+\theta)y^2\right)\, \mathrm d y \\
        =\left\{\begin{array}{ll}
       \ds  \sqrt{\pi\over 1+\theta}\left(-\theta\over 1+\theta\right)^{\frac{j+k}2}{{(-1)^k}}\sum_{i=0}^{\lfloor k/2 \rfloor}{j!k!\over i!(k-2i)! [(j-k)/2+i]!}\left(2\over \theta\right)^{k-2i},
 & \text{if $(j+k)$ is even,} \vspace{5pt}\\
            0, & \text{if $(j+k)$ is odd.}
        \end{array}\right.
\end{multline*}
\end{lem}

\begin{proof}
Changing the variable we have
    \begin{eqnarray*}
      &&  \int_\R H_j(y)H_k(y)\exp\left(-y^2-\theta (y-\kappa)^2\right)\, \mathrm d y\\
            &&\quad ={1\over \sqrt{1+\theta}}\exp\left(-\, {\theta\kappa^2\over 1+\theta}\right)\int_\R H_j\left(\frac x {\sqrt{1+\theta}}+\frac {\theta\kappa} {{1+\theta}}\right)H_k\left(\frac x {\sqrt{1+\theta}}+\frac {\theta\kappa} {{1+\theta}}\right)e^{-x^2}\, \mathrm d x\\
        &&\quad ={(1+\theta)^{-(j+k+1)/2}}\theta^{(j+k)/2}\exp\left(-\, {\theta\kappa^2\over 1+\theta}\right)\\
        &&\qquad \times \sum_{l_1=0}^j\sum_{l_2=0}^k\left(\begin{matrix}j\\l_1\end{matrix}\right)\,\left(\begin{matrix}k\\l_2\end{matrix}\right)\, \theta^{-(l_1+l_2)/2} 
      (H_{j-l_1}H_{k-l_2})\left(\sqrt{\theta\over 1+\theta}\kappa\right)\int_\R H_{l_1}(x)H_{l_2}(x)e^{-x^2}\mathrm d x,
\end{eqnarray*}
thanks to a double use of \eqref{mult_hermite_pol} (with $\gamma=\frac{1}{\sqrt{1+\theta}}$, $y=\frac{\theta \kappa}{\sqrt{1+\theta}}$). The orthogonality of the family $(H_i e^{-\, \bullet^2/2})_{i\in \N}$, more precisely \eqref{eq:eigenfunc_hermite}, completes the proof for a general $\kappa$. The case $\kappa =0$ is then a direct consequence of $H_{2k}(0)=(-1)^k{(2k)!\over k!}$ and $H_{2k+1}(0)=0$.
\end{proof}

\section{Proof of Lemma \ref{lem:taupe-plus}}\label{B:proof-taupe}

Recall that  $\lambda=\sqrt{\ot\over 1+\ot}$. In order to prove Lemma \ref{lem:taupe-plus}, in view of \eqref{rename-gamma}, we may change the definition of  $\gamma_j^k$  as
$$
\gamma_j^k= \frac{1}{(1+j)^b}\lambda^j\lambda^k \sum_{l=0}^{\min(k,j)}\frac{\sqrt{\binom{j}{l}\binom{k}{l}}}{\max(1,(k-l)^{1/4})\times \max(1,(j-l)^{1/4})}(\ot^{-1})^l.
$$
We set
$$\tg_j^k:= \frac{1}{(1+j)^b}\lambda^j\lambda^k \sum_{l=0}^{\min(k,j)}\sqrt{\binom{j}{ l }\binom{k}{l}}(\ot^{-1})^l.
$$
In what follows, $C$ denotes a generic positive constant, independent of $j$ and $k$, though its value may change from one occurrence to the next. Observe also that, the assumptions in Theorem \ref{th:lin_pursuit} and the definition of $\ot$ in \eqref{shortcuts} insure that $\ot\in(0,\sqrt{5}-2)$.

\noindent \paragraph{(i) We first show that, for some $0<q<1$, $\ds\sum_{0\le j\le \frac{k}{2}}\gamma_j^k=\mathcal O(q^k)$ as $k\to+\infty$.} For $0\le j \le \frac{k}{2}$,  using the crude estimates $\ds \binom{j}{l}\le 2^j$ and $\ds \binom{k}{l}\le 2^k$, we get
\begin{align*}
   \ds \gamma_j^k  \le \tg_j^k \le (\sqrt{2}\lambda)^{j+k}\sum_{l=0}^{j}(\ot^{-1})^l= (\sqrt{2}\lambda)^{j+k}\frac{(\ot^{-1})^{j+1}-1}{\ot^{-1}-1}\le C (\sqrt{2}\lambda)^{j+k} (\ot^{-1})^{j}.
\end{align*}
Summing over $0 \le j\le  k/2 $ (if $k/2$ is not an integer, this means that we sum over $0 \le j\le \lfloor k/2 \rfloor $), we obtain
\begin{align*}
   \ds \sum_{0\le j \le  \frac{k}{2}} \gamma_j^k  \le  C (\sqrt{2}\lambda)^{k} \sum_{0\le j \le \frac{k}{2}}(\sqrt{2}\lambda \ot^{-1})^{j}\le C (\sqrt{2}\lambda)^{k}(\sqrt{2}\lambda \ot^{-1})^{k/2}=C\lp\frac{2\sqrt{2}\lambda}{1+\ot} \rp^{k/2}.
\end{align*}
As $\ot<\sqrt{5}-2$, we have $\ds\frac{2\sqrt{2}\lambda}{1+\ot} <1$, and the conclusion of point (i) follows.

\noindent \paragraph{(ii) We next show a similar estimate for $\ds\sum_{j\ge \frac{3k}{2}}\gamma_j^k$.} For $j\ge \frac{3k}{2},$ using $\ds \binom{j}{l}\le 2^j$ and Cauchy-Schwarz inequality, we get 
\begin{align*}
    \gamma_j^k & \le \tg_j^k \le \frac{\lambda^{j+k}}{(1+j)^b}\sqrt{2}^j \sum_{l=0}^k \sqrt{\binom{k}{l}}(\ot^{-1})^l \nonumber \\
    & \le \frac{\lambda^{j+k}}{(1+j)^b}\sqrt{2}^j  \lp \sum_{l=0}^k \binom{k}{l}(\ot^{-1})^l \sum_{l=0}^k (\ot^{-1})^l\rp^{1/2} \nonumber\\
    & =\frac{\lambda^{j+k}}{(1+j)^b}\sqrt{2}^j  \lp \lp \frac{1+\ot}{\ot}\rp^k \frac{(\ot^{-1})^{k+1}-1}{\ot^{-1}-1}\rp^{1/2} \nonumber\\
    & \le \frac{C}{(1+j)^b} (\sqrt{2}\lambda)^j (\ot^{-1})^{\frac{k+1}{2}} \le C (\ot^{-1})^{\frac{k}{2}} (\sqrt{2}\lambda)^j. 
\end{align*}
Summing over $j\ge 3k/2$, we obtain
\begin{align*}
    \ds\sum_{j\ge \frac{3k}{2}}\gamma_j^k \le  C (\ot^{-1})^{\frac{k}{2}} \frac{(\sqrt{2}\lambda)^{\frac{3k}{2}}}{1 - \sqrt{2}\lambda} = C \lp \frac{(\sqrt{2}\lambda)^3}{\ot}\rp^{k/2}=C \lp \frac{2\sqrt{2}\lambda}{1+\ot} \rp^{k/2},
\end{align*}
and the conclusion of point (ii) follows.

\noindent \paragraph{(iii) Last, we show that $\ds\sum_{\frac{k}{2}<j<\frac{3k}{2}}\gamma_j^k\le C\frac{1}{(1+k)^{b-\frac{1}{2}}}$.} We begin with an inequality on the product of binomial coefficients.

\begin{lem}\label{lem:ineqbino}
Let  $j, \, k,\, l \in \N$ such that $l\le \min(j,k)$. Then
$$
 \binom{j}{l}\binom{k}{l}\le \binom{\frac{j+k}{2}}{l}^2.
 $$
\end{lem}

\begin{proof}
    We define the function $\ds f: \ x\mapsto \binom{x}{l}:=\frac{\Gamma(x+1)}{\Gamma(l+1)\Gamma(x-l+1)}$ for $x> l-1$. Then, $$(\ln f)''(x)=\Psi'(x+1)-\Psi'(x-l+1),$$with $\ds \Psi:=\frac{\Gamma'}{\Gamma}$ the digamma function. As $\Psi$ is  concave on $(0,+\infty)$, $(\ln f)''(x) \le 0$ for all $x> l-1$. This implies that $\ln f$ is concave on $(l-1,+\infty)$, which in turns implies that $$
    \ln f\lp \frac{j+k}{2}\rp \ge \frac{1}{2}\ln f (j)+\frac{1}{2}\ln f (k),
    $$
    which provides the result.
\end{proof}

\Bk

Using Lemma \ref{lem:ineqbino}, we obtain
\begin{align*}
    \tg_j^k & \le  \frac{1}{(1+j)^b}\lambda^j\lambda^k \sum_{l=0}^{\min(k,j)}\binom{\frac{j+k}{2}}{ l }(\ot^{-1})^l  \\
    & \le  \frac{1}{(1+j)^b}\lambda^j\lambda^k \sum_{l=0}^{\frac{j+k}{2}}\binom{\frac{j+k}{2}}{ l }(\ot^{-1})^l \\
    & =\frac{1}{(1+j)^b} \lp \frac{\ot}{1+\ot}\rp^{\frac{j+k}{2}}\lp \frac{1+\ot}{\ot}\rp^{\frac{j+k}{2}} = \frac{1}{(1+j)^b}.
\end{align*}
As $\frac 12 k\leq j\leq \frac 32 k$, we get
\begin{equation} \label{eq:ineg_gamtilde}
    \tg_j^k \le \frac{2^b}{(1+k)^b} \le  \frac{C}{(1+k)^b}.
\end{equation}
In the sequel, we define
$$
a_{j,k,l}:=\sqrt{\binom{j}{l}\binom{k}{l}}(\ot^{-1})^l,
$$
so that  we may write
$$
\gamma_j^k= \frac{1}{(1+j)^b}\lambda^j\lambda^k \sum_{l=0}^{\min(k,j)}\frac{a_{j,k,l}}{\max(1,(k-l)^{1/4})\max(1,(j-l)^{1/4})},
$$
and
$$
\tg_j^k= \frac{1}{(1+j)^b}\lambda^j\lambda^k \sum_{l=0}^{\min(k,j)}a_{j,k,l}.
$$
Using Cauchy–Schwarz  inequality, we note that 
\begin{equation} \label{eq:jensen}
    \sum_{l=0}^{\min(k,j)-1}\frac{a_{j,k,l}}{(k-l)^{1/4} (j-l)^{1/4}} \le \lp\sum_{l=0}^{\min(k,j)-1}a_{j,k,l} \rp^{1/2}\lp \sum_{l=0}^{\min(k,j)-1}\frac{a_{j,k,l}}{(k-l)^{1/2} (j-l)^{1/2}}  \rp^{1/2}.
\end{equation}
For $l\le \frac{k}{4}$, we have $(k-l)^{1/2} (j-l)^{1/2}\ge(k-k/4)^{1/2} (k/2-k/4)^{1/2} \ge C k$. Thus, 
\begin{equation} \label{eq:sum1}
    \sum_{l=0}^{\lfloor k/4 \rfloor}\frac{a_{j,k,l}}{(k-l)^{1/2} (j-l)^{1/2}}\le \frac{C}{k} \lp\sum_{l=0}^{\lfloor k/4 \rfloor}a_{j,k,l} \rp.
\end{equation}
Next, for $l \ge \lfloor k/4 \rfloor +1$ and $l\le \min(j,k)-1$ we have
$$  \sqrt{\frac{\binom{j}{l} }{j-l}}  =  \sqrt{\frac{\binom{j}{l-1} }{j-l} \frac{j-l+1}{l}} = \sqrt{\binom{j}{l-1} \frac{1}{l} \lp1+ \frac{1}{j-l} \rp }\le \frac{C}{\sqrt{k}} \sqrt{\binom{j}{l-1}},
$$
and, similarly,
$$  \sqrt{\frac{\binom{k}{l} }{k-l}} \le \frac{C}{\sqrt{k}} \sqrt{\binom{k}{l-1}}.
$$
Thus, 
\begin{equation} \label{eq:sum2}
    \sum_{l=\lfloor k/4 \rfloor+1}^{\min(k,j)-1}\frac{a_{j,k,l}}{(k-l)^{1/2} (j-l)^{1/2}}\le \frac{C}{k} \lp \sum_{l=\lfloor k/4 \rfloor+1}^{\min(k,j)-1} a_{j,k,l-1} \rp= \frac{C}{k} \lp \sum_{l=\lfloor k/4 \rfloor}^{\min(k,j)-2} a_{j,k,l} \rp.
\end{equation}
Adding \eqref{eq:sum1} and \eqref{eq:sum2}, we obtain
\begin{equation} \label{eq:sum3}
    \sum_{l=0}^{\min(k,j)-1}\frac{a_{j,k,l}}{(k-l)^{1/2} (j-l)^{1/2}}\le \frac{C}{k} \lp a_{j,k,\lfloor k/4 \rfloor}+\sum_{l=0}^{\min(k,j)-2} a_{j,k,l} \rp \le \frac{C}{k} \lp \sum_{l=0}^{\min(k,j)-1} a_{j,k,l} \rp .
\end{equation}
Plugging this into \eqref{eq:jensen}, we get
\begin{equation} \label{eq:jensen2}
    \sum_{l=0}^{\min(k,j)-1}\frac{a_{j,k,l}}{(k-l)^{1/4} (j-l)^{1/4}} \le  \frac{C}{\sqrt{k}} \lp \sum_{l=0}^{\min(k,j)-1}a_{j,k,l}  \rp.
\end{equation}
Coming back to $\gamma_j^k$, we have
\begin{align*}
    \gamma_j^k & = \frac{1}{(1+j)^b}\lambda^j\lambda^k \sum_{l=0}^{\min(k,j)}\frac{a_{j,k,l}}{\max(1,(k-l)^{1/4}) \max(1,(j-l)^{1/4})} \\
    & \le \frac{1}{(1+j)^b}\lambda^j\lambda^k \left[a_{j,k,\min(j,k)}+\sum_{l=0}^{\min(k,j)-1}\frac{a_{j,k,l}}{(k-l)^{1/4}(j-l)^{1/4}}\right] \\
    & \le \frac{1}{(1+j)^b}\lambda^j\lambda^k \left[a_{j,k,\min(j,k)}+\frac{C}{\sqrt{k}} \lp \sum_{l=0}^{\min(k,j)-1}a_{j,k,l} \rp\right]
    \\
    & \le \frac{C}{\sqrt{k}} \frac{1}{(1+j)^b}\lambda^j\lambda^k \lp \sum_{l=0}^{\min(k,j)-1}a_{j,k,l} \rp \\ & \le \frac{C}{\sqrt{k}} \tg_j^k,
\end{align*}
where we have used \eqref{eq:jensen2} and $a_{j,k,\min(j,k)} \le \frac{C}{\sqrt{k}}a_{j,k,\min(j,k)-1}$ for $\frac 12 k\leq j\leq \frac 32 k$ and $k\ge 1$. With \eqref{eq:ineg_gamtilde}, we obtain 
\begin{equation}
     \gamma_j^k \le \frac{C}{(1+k)^{b+1/2}}.
\end{equation}
Finally, this shows that 
$\ds\sum_{\frac{k}{2}<j<\frac{3k}{2}}\gamma_j^k\le C\frac{1}{(1+k)^{b-\frac{1}{2}}}$, which concludes the proof of point (iii) and of Lemma \ref{lem:taupe-plus}. \qed

\paragraph*{Acknowledgement.}
Matthieu Alfaro is supported by  the \textit{région Normandie} project BIOMA-NORMAN 21E04343 and the ANR project DEEV ANR-20-CE40-0011-01. Florian Lavigne would like to acknowledge the \textit{région Normandie} for the financial support of his Post-Doctoral position.

\bibliographystyle{siam}
\bibliography{biblio_clean}

\end{document}